\documentclass{amsart}

\usepackage{graphicx} 
\usepackage{url}

\numberwithin{equation}{section}

\usepackage{amssymb,mathtools}

\usepackage{multicol}

\newtheorem{theorem}{Theorem}[section]
\newtheorem{lemma}[theorem]{Lemma}
\newtheorem{assumption}[theorem]{Assumption}
\newtheorem{corollary}[theorem]{Corollary}

\newcommand{\Id}{\mathrm{Id}}

\newcommand{\propar}{P_\parallel}
\newcommand{\properp}{P_\perp}

\newcommand{\tn}[1][n]{t_{#1}}

\newcommand{\normempty}[1]{\left\lVert #1 \right\rVert}

\newcommand{\psifilter}{\textcolor{blue}{\Psi}}
\newcommand{\phifilter}{\textcolor{blue}{\Phi}}
\newcommand{\Chi}{X}
\newcommand{\chifilter}{\textcolor{blue}{\Chi}}
\newcommand{\upsilonfilter}{\textcolor{blue}{\Upsilon}}
\newcommand{\funexpGauss}{\varrho}

\newcommand{\linop}{L}
\newcommand{\mgFieldStrength}{B}
\DeclareMathOperator{\sinc}{sinc}
\DeclareMathOperator{\sinch}{sinch}
\DeclareMathOperator{\tanch}{tanch}

\usepackage{xcolor}

\newcommand{\pos}[1][]{{q_{#1}}}
\newcommand{\vel}[1][]{{v_{#1}}}
\newcommand{\rem}{w}
\newcommand{\tpos}[1][]{{\widetilde{q}_{#1}}}
\newcommand{\tvel}[1][]{{\widetilde{v}_{#1}}}

\newcommand{\nonres}{\vartheta}
\newcommand{\defecterr}[2][]{{\Sigma_{#1}^{#2}}}

\newcommand{\elField}[1][]{{E_{#1}}}
\newcommand{\telField}[1][]{{\widetilde{E}_{#1}}}
\newcommand{\telFieldJac}[2][]{{\widetilde{E}_{#1}^{(#2)}}}
\newcommand{\elFieldJac}{E'}
\newcommand{\bdelFieldJac}[2][]{{M^{(#2)}_{\elField,#1}}} 
\newcommand{\remainelField}[2][]{{r^{(#2)}_{\elField,#1}}} 

\newcommand{\mgField}[1][]{{B_{#1}}}

\newcommand{\abs}[1]{\lvert #1 \rvert}

\NewDocumentCommand{\ts}{o}{\IfValueTF{#1}{t_{#1}}{t}}

\NewDocumentCommand{\mgFieldMatrix}{s o}{\IfBooleanTF{#1}{
		\mathbf{B}(\ts[#2], \pos[#2])
	}{
		\mathbf{B}\IfValueT{#2}{_{#2}}
}}

\NewDocumentCommand{\systemMatrix}{s o m o}{\IfBooleanTF{#1}{
		\Id #3 \tfrac{\tau}{2} \mgFieldMatrix[#4]
	}{
		\mathbf{R}_{#3\IfValueT{#4}{,#4}}\IfValueT{#2}{^{#2}}
}}

\newcommand{\auxMatrix}{\mathbf{S}}

\NewDocumentCommand{\defect}{s o}{\IfBooleanT{#1}{\widetilde}d_{\IfValueT{#2}{#2}}}

\NewDocumentCommand{\errorelFieldVec}{s o}{\IfBooleanT{#1}{\widetilde} g_{\IfValueT{#2}{#2}}}

\NewDocumentCommand{\defectPos}{s o}{\IfBooleanT{#1}{\widetilde}d_{q\IfValueT{#2}{,#2}}}

\NewDocumentCommand{\defectVel}{s o}{\IfBooleanT{#1}{\widetilde}d_{v\IfValueT{#2}{,#2}}}

\NewDocumentCommand{\errorPos}{s o}{\IfBooleanTF{#1}{
		\tpos[#2] - \pos[#2]
	}{
		e_{q\IfValueT{#2}{,#2}}
}}

\NewDocumentCommand{\error}{s o}{\IfBooleanTF{#1}{\widetilde}e_{\IfValueT{#2}{#2}}}

\NewDocumentCommand{\errorVel}{s o}{\IfBooleanTF{#1}{
		\tvel[#2] - \vel[#2]
	}{
		e_{v\IfValueT{#2}{,#2}}
}}

\NewDocumentCommand{\errorelField}{s o}{\IfBooleanTF{#1}{
		\telField[#2] - \elField[#2]
	}{
		e_{\elField\IfValueT{#2}{,#2}}
}}

\newcommand{\lipschitzelField}{L_{\elField}}

\begin{document}

\title[Boris-type exponential integrators]{Boris-type exponential integrators for charged-particle dynamics in strong magnetic fields}

\date{}

\author{Marlis Hochbruck}
\address{Institute for Applied and Numerical Mathematics,
	Karlsruhe Institute of Technology, Englerstra\ss e 2, 76131 Karlsruhe, Germany}
\email{marlis.hochbruck@kit.edu}
\thanks{Funded by the Deutsche Forschungsgemeinschaft (DFG, German Research Foundation) -- Project-ID 258734477 -- SFB 1173.}

\author{Sebastian Merk}
\address{Department of Mathematics, Technical University of Munich,
	Boltzmannstra\ss e 3, 85748 Garching bei M\"unchen, Germany}
\email{s.merk@tum.de}
\thanks{}


\begin{abstract}
We study numerical time integration for the motion of a charged particle in a strong magnetic field, focusing on the regime in which the fast gyration is not resolved by the time step. Standard Boris-type methods are structure preserving and accurate when the gyration is resolved, but their error constants deteriorate with increasing field strength; filtered Boris methods improve this behavior but may suffer from singularities and resonance-induced error blow-up. In this work we consider the case of a constant strong magnetic field and derive a broad family of one-step exponential integrators from the variation-of-constants formula, formulated in terms of matrix-valued filter functions. This framework includes existing filtered Boris variants and permits the construction of schemes with uniformly bounded filters. Using a discrete variation-of-constants representation together with summation-by-parts arguments, we establish error bounds whose constants are independent of the magnetic field strength and characterize the filter conditions required for first- and second-order accuracy. In particular, we show that bounded filters avoid resonance blow-up and allow arbitrary time steps, at the cost of a controlled order reduction near resonant frequencies. Numerical experiments confirm the predicted convergence behavior, the different accuracy of position and velocity components, and the role of the leading error terms across a wide range of step-size-field-strength products.
\end{abstract}

\maketitle

\section{Introduction}\label{sec:intro}
We study time integration methods for the Euler--Lagrange system of a charged
particle in a strong magnetic field
\begin{subequations} \label{eq:eul-lag-filt-all}
	\begin{align}
		\dot{\pos}(t) & = \vel(t),
		\label{eq:eul-lag-filt-a}
		& \pos(0) = \pos[0],\\
		\dot{\vel}(t) &= \vel(t)\times \mgField(t,\pos) + \elField(t,\pos),
		& \vel(0) = \vel[0],
		\label{eq:eul-lag-filt-b}
	\end{align}
\end{subequations}
where $\normempty{\mgField[]} \gg 1$. The particle gyrates about the magnetic field lines with period $2\pi\delta$, where $\delta = 1/\normempty{\mgField[]}$, and this gyration is superimposed on a motion of the guiding center that takes place on the slow, $\mathcal{O}(1)$ scale fixed by $\elField[]$. 

Systems of the form \eqref{eq:eul-lag-filt-all} arise as the characteristic equations of the Vlasov equation and, in particular, as the particle push in numerical solutions
of the Vlasov--Maxwell system \cite[Sec. 15]{Bir18}, and more recently in variational approximations of the magnetic Schr\"odinger equation \cite{SchBHL25}, that could be used to solve coupled Maxwell--Schr\"odinger systems. In both settings, the electromagnetic field is not given but computed self-consistently, so that the field strength experienced by a particle is not known a priori. Commonly used Boris-type pushers \cite{Boris70} are accurate only for stepsizes
$\tau \lesssim \delta$ resolving the gyration. In the coupled setting, this forces $\tau$ to be chosen from a worst-case field strength, or requires adaptivity of the time step, which sacrifices the structure preservation that motivates the use of Boris-type methods.

We recall that the scheme in a formulation on a staggered grid reads
\begin{subequations}\label{eq:boris-leapfrog}
	\begin{align}
		\frac{\pos[n+1] - \pos[n]}{\tau}
		&=
		\vel[n+1/2],
		\label{eq:boris-leapfrog-a}
		& \mgField[n] &= \mgField(\ts[n], \pos[n]),
		\\[1mm]
		\frac{\vel[n+1/2] - \vel[n-1/2]}{\tau}
		&=
		\tfrac12 \bigl( \vel[n+1/2] + \vel[n-1/2] \bigr) \times \mgField[n] + \elField[n],
		& \elField[n] &= \elField(\ts[n], \pos[n]).
		\label{eq:boris-leapfrog-b}
	\end{align}
\end{subequations}
It is symmetric and volume-preserving \cite{QinZXLST13}, and it shows no drift in the energy over long times \cite{HaiL18}. It is of second order, but the corresponding error constants grow with $\normempty{\mgField[]}$, and the accuracy of \eqref{eq:boris-leapfrog} deteriorates as soon as $\tau \normempty{\mgField[]} \gtrsim 1$, cf. \cite[Fig.~8.2]{HaiLW20}.

This behavior is not particular to \eqref{eq:boris-leapfrog}. Problems of the form \eqref{eq:eul-lag-filt-all} are stiff in the second sense discussed in \cite{HocO10}: the linear part of \eqref{eq:eul-lag-filt-b} has purely imaginary eigenvalues $0$ and $\pm i \normempty{\mgField[]}$ of large modulus. An explicit method applied to \eqref{eq:eul-lag-filt-all} is stable only for $\tau \lesssim \delta$; implicit methods are stable for large stepsizes, but their accuracy deteriorates once the gyration is no longer resolved. The difficulty also affects the analysis. The usual convergence theory relies 
on Taylor expansion of the exact solution, whose derivatives grow with powers of $\normempty{\mgField[]}$, cf. \eqref{eq:eul-lag-filt-b}, so that the stepsize restriction reappears in the error constants. Bounds of this type yield the classical order of a method, while we are interested in its stiff order, i.e., in error bounds whose constants are independent of $\normempty{\mgField[]}$. This requires both a different construction and a different analysis.

Closest to the present work are the filtered Boris methods. Hairer, Lubich, and Wang~\cite{HaiLW20} modify the velocity update \eqref{eq:boris-leapfrog-b} by introducing a suitable filter function. Hairer, Lubich, and Shi~\cite{HaiLS22} study a filtered variational integrator with the same filter, in the regime $\delta \leq \tau^2 \ll 1$. The two constructions differ in how the rotation and the approximations to the velocities on the full time grid are realized. For a constant magnetic field, these methods are equivalent. The error analysis in both references relies on modulated Fourier expansions and yields second-order bounds for the positions and for the velocity component parallel to $\mgField[]$, together with first-order bounds for the perpendicular component.

Filters of this kind have a longer history in the analysis of highly oscillatory second-order problems $\ddot{y} = -\Omega^2 y + g(y)$ with $\normempty{\Omega} \gg 1$. Garc\'ia, Sanz-Serna, and Skeel~\cite{GarSS98} introduced the mollified impulse method, in which the nonlinearity is evaluated at an averaged argument, and Hochbruck and Lubich~\cite{HocL99} analyzed a Gautschi-type integrator that is exact for linear problems with constant inhomogeneity. In both cases, the error bounds are uniform in $\tau\normempty{\Omega}$, and the analysis produces conditions on the filter functions rather than bounds for a single scheme: the filters must vanish at the
resonant frequencies, and their remaining freedom can be used to reduce the error constant. Grimm and Hochbruck~\cite{GriH06} formulate such conditions for a two-parameter family of trigonometric integrators and obtain second-order bounds for the positions together with first-order bounds for the velocities. The structure of these results, and the technique of summation-by-parts on which they rely, is what we transfer to \eqref{eq:eul-lag-filt-all}.

A second line constructs exponential integrators for \eqref{eq:eul-lag-filt-all} with an emphasis on geometric structure: explicit symmetric schemes of classical order up to four \cite{WuW20}, energy-preserving schemes for a constant strong field \cite{Wan21}, and
continuous-stage energy-preserving methods designed to operate from normal to strong field regimes \cite{LiW22}. Nguyen, Joseph, and Tokman \cite{NguJT24, NguJT25} pursue the same idea from a computational perspective, comparing exponential integrators with conventional pushers in a particle-in-cell context and deriving Nystr\"om-type exponential schemes of second and third order. These constructions are likewise built from filter functions, but the orders attained are classical, no bounds uniform in $\normempty{\mgField[]}$ are established, and no conditions on the matrix functions themselves are formulated.

Among the Boris-type methods mentioned above, bounds of this kind are so far available only for the single filter $\psi(\zeta) = \tanch(\zeta/2)$, which, unfortunately, has singularities
at integer multiples of $2\pi$. As a consequence, the velocity error of the scheme by Hairer, Lubich, and Wang \cite{HaiLW20} blows up like the reciprocal distance of $\tau\normempty{\mgField[]}$ to $2k\pi$, cf.\ Fig.~8.2 in their paper.
Moreover, their second-order bounds of the position error require both a coupling $\tau \normempty{\mgField[]} \lesssim 1$ of the stepsize to the inverse field strength and a nonresonance condition on $\tau \normempty{\mgField[]}$. 
However, if $\mgField[]$ is not given explicitly but via a Maxwell solver, it is not practical to ensure such a nonresonance condition, since it might require adapting the time stepsize during the simulation, which might lead to further disadvantages. 

This motivated us to reconsider the construction of such integrators and to provide a scheme, which only uses uniformly bounded filter functions. We will prove that the error close to resonances drops to order one, but does not lead to blow up at any time stepsize. Order reduction and blow up already appear for constant magnetic fields 
$\mgField[]$ and thus, as a first step towards an analysis of the general scheme, we restrict ourselves to constant magnetic fields  in the regimes described above.  $\normempty{\mgField[]} \gg 1$ is the only large parameter, while $\elField[] = \mathcal{O}(1)$ and $\vel_0 = \mathcal{O}(1)$, where here and throughout the whole paper, the notation $\mathcal{O}$ means that the constants are independent of $\normempty{\mgField[]}$. 

The extension to non-constant fields is left to future work.

We therefore consider a generic one-step exponential integrator 
\begin{subequations}  \label{eq:gen-expint} 
	\begin{align}  
		\label{eq:gen-expint-pos}
		\pos[n+1] &= \pos[n] 
		+ \tfrac{\tau}{2} \bigl( \psifilter_{-} \vel[n]+  \psifilter_{+} \vel[n+1] \bigr)
		+ \tfrac{\tau^2}{4} \bigl(
		\chifilter_{-}\elField[n]
		+  \chifilter_{+} \elField[n+1] \bigr),	\\
		\label{eq:gen-expint-vel}
		\vel[n+1] &= e^{\tau \mgFieldMatrix[]} \vel[n] + \tfrac{\tau}{2}
		\bigl( \phifilter_- \elField[n]  + \phifilter_{+} \elField[n+1] \bigr),
	\end{align}
\end{subequations}
with given filters defined on the eigenvalues $0, \pm i \tau \normempty{\mgField[]}$ of $\tau\mgFieldMatrix[]$, where $\mgFieldMatrix[]$ denotes the matrix representing $\vel\mapsto \vel \times \mgField[]$:
\begin{equation*}
	\phifilter_\pm = \phi_\pm(\tau \mgFieldMatrix[]),
	\qquad
	\psifilter_\pm = \psi_\pm(\tau \mgFieldMatrix[]),
	\qquad
	\chifilter_\pm = \chi_\pm(\tau \mgFieldMatrix[]).
\end{equation*}
This class contains the filtered Boris variants of Hairer, Lubich, and Wang~\cite{HaiLW20} and Hairer, Lubich, and Shi~\cite{HaiLS22} but also more general methods which are derived from the variation-of-constants formula and applying exponential quadrature rules up to order two.
In contrast to modulated Fourier expansions used so far, our error analysis is based on the (discrete) variation-of-constants formula and partial integration/summation. This allows us to fully characterize the filter functions, study the error constants, and the behavior of the integrator at resonant frequencies.  

The filter functions of \cite{HaiLW20} are singular at $\tau\normempty{\mgField[]} \in 2\pi\mathbb{N}$, and the velocity error blows up like the reciprocal distance of $\tau\normempty{\mgField[]}$ to $2k\pi$. Within our framework, filters without such singularities are available. For these, the first-order bounds of Theorem~\ref{thm:error-bound} hold without any nonresonance assumption, so that arbitrary stepsizes are admissible; the blow-up is avoided at the cost of an order reduction in the velocity error. 

The paper is organized as follows. Section~\ref{sec:filt-boris-alg} recalls the filtered Boris algorithms and establishes their equivalence for a constant magnetic field. Section~\ref{sec:exp-int} derives the family of exponential integrators from the variation-of-constants formula. Section~\ref{sec:error} sets up the error recursion and the stability estimates, Sections~\ref{sec:order-1} and~\ref{sec:order-2} contain the first- and second-order error bounds and the refined analysis of the velocity error. In Section~\ref{sec:numerics}, we present numerical experiments which confirm the predicted orders in the different regimes of $\tau\normempty{\mgField[]}$ and the distinct behavior of the error components, and also show that the leading error terms describe the observed error over the whole range of $\tau\normempty{\mgField[]}$.

\section{Review on filtered Boris algorithms} \label{sec:filt-boris-alg}

We first review the recent work by Hairer, Lubich, Shi, and Wang \cite{HaiLS22,HaiLW20} on filtered Boris-type integrators. The algorithm from \cite{HaiLS22} written in one-step formulation on a staggered time grid is given as 
\begin{multicols}{2}
\begin{subequations} \label{eq:filt-boris-all}
	\noindent
	\begin{align}
		\vel[-] 
		&= 
		\vel[n-\frac12] + \frac{\tau}{2} \psifilter \elField[n], 
		\label{eq:filt-boris-a}
		\\
		\vel[+] -\vel[-] 
		&= 
		\frac{\tau}{2} \psifilter \bigl( \vel[+] +\vel[-] \bigr) \times \mgField[n],
		\label{eq:filt-boris-b}
	\end{align}
	\columnbreak
	\begin{align}
		\vel[n+\frac12] 
		&= 
		\vel[+] + \frac{\tau}{2} \psifilter \elField[n], 
		\label{eq:filt-boris-c}
		\\
		\pos[n+1] 
		&= 
		\pos[n] + \tau\vel[n+\frac12].
		\label{eq:filt-boris-d}
	\end{align}
\end{subequations}
\end{multicols}
\vspace*{-5ex}
\noindent
It reduces to the classical Boris algorithm for $\psi = 1$. \cite{HaiLS22} and previously \cite{HaiLW20} suggest to use the filter function 
\begin{equation} \label{eq:filter-HL}
\psifilter = \psi(\tau \mgFieldMatrix[n]), 
\qquad 
\psi(\zeta) = \tanch \bigl(\tfrac{\zeta}{2}\bigr) = \frac{\tanh\bigl(\tfrac{\zeta}{2}\bigr)}{\tfrac{\zeta}{2}},
\end{equation}
where the matrix $\mgFieldMatrix[n]$ is defined via
\begin{equation*}
\mgFieldMatrix[n] v = v \times \mgFieldStrength_n,
\qquad
\mgFieldMatrix[n] =
\begin{pmatrix}
	0 & b_3 & -b_2\\
	-b_3 & 0 & b_1\\
	b_2 & -b_1 & 0
\end{pmatrix},
\quad
\mgField[n]
=
\begin{pmatrix}
	b_1 \\ b_2 \\ b_3
\end{pmatrix}.
\end{equation*}
The filtered Boris algorithm can be written as in \eqref{eq:boris-leapfrog-b} by adding and subtracting \eqref{eq:filt-boris-a} and \eqref{eq:filt-boris-c}. 
This shows that \eqref{eq:filt-boris-a}--\eqref{eq:filt-boris-c} can be replaced by
\begin{equation}  \label{eq:filt-boris-v-Psi}
\vel[n+\frac12] - \vel[n-\frac12] = \frac{\tau}{2} \psifilter \bigl( \vel[n+\frac12] + \vel[n-\frac12] \bigr) \times\mgField[n]
+ \tau \psifilter \elField[n].
\end{equation}
Interestingly, this update can be interpreted as the exponential midpoint  rule \cite{HocO10} applied to \eqref{eq:eul-lag-filt-b}. To verify this, we write
\begin{equation} \label{eq:tanch-id}
\psi(\zeta) = \tanch (\zeta/2) 
= \frac{2}{\zeta} \, \frac{e^{\zeta/2}-e^{-\zeta/2}}{e^{\zeta/2}+e^{-\zeta/2}}
= \frac{2}{\zeta} \, \frac{e^{\zeta}-1}{e^{\zeta}+1}.
\end{equation}
Using this relation in \eqref{eq:filt-boris-v-Psi} and multiplying by $(e^{\tau \mgFieldMatrix[n]} + \Id)/2$ yields the equivalent form
\begin{equation} \label{eq:filt-boris-v}
\vel[n+1/2] = e^{\tau\mgFieldMatrix[n]} \vel[n-\frac12] + \tau \varphi_1(\tau\mgFieldMatrix[n]) \elField[n].
\end{equation}
Here, the $\varphi$-functions are given by
\begin{equation}
\varphi_k(\zeta) = \int_0^1 e^{ (1- \sigma) \zeta} \frac{\sigma^{k-1}}{(k-1)!} d\sigma,
\qquad k=0,1,\ldots
\end{equation}

The filtered version presented in \cite{HaiLW20} differs in the definition of $\vel[+]$. Instead of \eqref{eq:filt-boris-b}, it uses $\vel[+] =\exp(\tau \mgFieldMatrix[n]) \vel[-]$. Approximations to $\vel(t_{n})$ are computed as
\begin{align}
\label{eq:filt-boris-HLW-vn}
\vel[n] &= \bigl(\sinch(\tau\mgFieldMatrix[n])\bigr)^{-1} \frac{\pos[n+1]-\pos[n-1]}{2\tau}  - \tau \Upsilon(\tau\mgFieldMatrix[n]) \elField[n].
\end{align}
In the special case of a constant magnetic field $\mgField$, the integrator is equivalent to \eqref{eq:filt-boris-all}.
Lemma~\ref{lem:HWL-psipm-phi2-identity} shows that the scheme in \cite[eq.~(2.12)]{HaiLW20}\footnote{Note that a factor two is missing in the last sentence of Section~2 of this reference} is equivalent to
\begin{subequations} \label{eq:HLW-all}
\begin{align} 
	\label{eq:HLW_q}
	\pos[n+1] & = \pos[n] + \tau \varphi_1(\tau \mgFieldMatrix[])\vel[n] + \tau^2\varphi_2(\tau \mgFieldMatrix[])  \elField[n]\\
	\label{eq:voc-v-update-full=HLW}
	\vel[n+1]   &= e^{\tau \mgFieldMatrix[]}\vel[n] + \tau \varphi_1(-\tau \mgFieldMatrix[])^{-1}
	\bigl( \varphi_2(\tau \mgFieldMatrix[])\elField[n] + \varphi_2(-\tau \mgFieldMatrix[])\elField[n+1] \bigr),
\end{align}
\end{subequations}
if $\psi$ is chosen as in \eqref{eq:filter-HL}.

\section{Exponential integrators for the Euler--Lagrange equations} \label{sec:exp-int}

From now on, we consider the situation that we have a strong, constant magnetic field 
$\mgFieldStrength$, i.e., $\| \mgFieldStrength\| = 1/\delta \gg 1$ and 
a sufficiently regular, autonomous electrical field $\elField[]= \elField(\pos)$. 
In particular, we assume $\elField[]$ to be locally Lipschitz continuous.
Then, the Euler--Lagrange equations \eqref{eq:eul-lag-filt-all} can also be stated as an autonomous, semilinear system of ODEs
\begin{equation}
\dot y = \linop y + f(y),
\qquad 
y = 
\begin{pmatrix}
	\pos \\ \vel
\end{pmatrix},
\quad
\linop = 
\begin{pmatrix}
	0 & \Id \\
	0 & \mgFieldMatrix[]
\end{pmatrix},
\quad
f(y) = 
\begin{pmatrix}
	0 \\ E(\pos)
\end{pmatrix}.
\end{equation}
The variation-of-constants formula yields
\begin{equation} \label{eq:voc}
y(t+\tau) = e^{\tau \linop} y(t) + \tau \int_0^1 e^{\tau(1-\sigma)\linop} f\bigl( y(t+\tau\sigma) \bigr) d\sigma.
\end{equation}
Due to the special structure of $\linop$, we have the following representation of the matrix exponential, see also eq.~(2.13) in \cite{HaiLW20}
\begin{equation} \label{eq:expmatrix}
e^{\tau\linop} 
= 
\begin{pmatrix}
	\Id & \tau \varphi_1(\tau \mgFieldMatrix[])\\
	0 & e^{\tau \mgFieldMatrix[]}
\end{pmatrix}.
\end{equation}

By \eqref{eq:voc}, the components $\pos,\vel$ satisfy
\begin{subequations}  \label{eq:exact-q-v-voc}
\begin{align}
	\pos\bigl(t \pm \theta\bigr) 
	&= \pos(t) 
	\pm \theta \varphi_1\bigl(\pm\theta \mgFieldMatrix[]\bigr) \vel(t)
	+\theta^2 \int_0^1 (1-\sigma)
	\varphi_1\bigl(\pm \theta(1-\sigma)\mgFieldMatrix[]\bigr) 
	\elField\bigl( \pos(t\pm\theta\sigma) \bigr)
	d\sigma,	
	\label{eq:exact-q-v-voc-a}
	\\
	\vel(t\pm\theta) &= e^{\pm\theta \mgFieldMatrix[]} \vel(t) 
	\pm \theta \int_0^1 e^{\pm\theta(1-\sigma)\mgFieldMatrix[]} 
	\elField\bigl( \pos(t\pm\theta\sigma) \bigr) d\sigma.
	\label{eq:exact-q-v-voc-b}
\end{align}
\end{subequations}
The iteration \eqref{eq:filt-boris-v} is obtained by setting $t = \ts[n-1/2]$, $\theta = \tau$ in \eqref{eq:exact-q-v-voc-b}, and approximating 
\begin{equation*}
\elField\bigl( \pos(t+\tau\sigma) \bigr) \approx \elField[n] = \elField(\pos[n]),
\qquad \sigma \in [0,1].
\end{equation*}
Setting $t=\ts[n+1/2]$ and $\theta = \pm\tau/2$ in \eqref{eq:exact-q-v-voc-a} and approximating the electric field by suitable constants $\elField[n,\pm]$ yields the approximations
\begin{subequations}\label{eq:posnp1}
\begin{align}
	\label{eq:posnp1-a}
	\pos[n+1] & = \pos[n+1/2] 
	+ \tfrac{\tau}{2} \varphi_1\bigl(\tfrac{\tau}{2} \mgFieldMatrix[]\bigr) \vel[n+1/2]
	+ \tfrac{\tau^2}{4} \varphi_2\bigl(\tfrac{\tau}{2}\mgFieldMatrix[]\bigr) 
	\elField[n,+], \\
	\label{eq:posnp1-b}
	\pos[n] & = \pos[n+1/2] 
	- \tfrac{\tau}{2} \varphi_1\bigl(-\tfrac{\tau}{2} \mgFieldMatrix[]\bigr) \vel[n+1/2]
	+ \tfrac{\tau^2}{4} \varphi_2\bigl(-\tfrac{\tau}{2}\mgFieldMatrix[]\bigr) 
	\elField[n,-].
\end{align}
A natural choice would be to set
\begin{equation} \label{eq:Enpm}
	\elField[n,-] = \elField[n], \qquad
	\elField[n,+] = \elField[n+1].
\end{equation}
\end{subequations}
This makes the scheme implicit, but it should be possible to solve it by just a few steps of fixed-point iteration because we assumed that $\elField[]$ does not impose stiffness to the differential equation. With this choice, subtracting  \eqref{eq:posnp1-b} from \eqref{eq:posnp1-a} yields with \eqref{eq:id-sinch}
\begin{equation} \label{eq:voc-q-update}
\pos[n+1] = \pos[n] 
+ \tau \sinch \bigl(\tfrac{\tau}{2} \mgFieldMatrix[]\bigr)
\vel[n+1/2]
+ \tfrac{\tau^2}{4} \Bigl(
\varphi_2\bigl(\tfrac{\tau}{2}\mgFieldMatrix[]\bigr) \elField[n+1]
-  \varphi_2\bigl(-\tfrac{\tau}{2}\mgFieldMatrix[]\bigr) \elField[n] \Bigr).
\end{equation}
With the choice \eqref{eq:Enpm} and the update \eqref{eq:filt-boris-v}, we obtain a symmetric method which is exact for constant magnetic and electric fields $\mgField[]$ and $\elField$. The update \eqref{eq:filt-boris-d} used in the classical and the filtered Boris algorithm  corresponds to approximating
$\sinch\bigl(\tfrac{\tau}{2} \mgFieldMatrix[]\bigr) \approx \Id$ and neglecting the terms involving the $\varphi_2$-functions in \eqref{eq:voc-q-update}.

Directly using the variation-of-constants formula \eqref{eq:exact-q-v-voc-b} yields a different scheme than \cite[eq.~(2.12)]{HaiLW20} for the $v$-component on the full grid. A  first simple choice, which is possibly of order two, would be 
\begin{subequations}  \label{eq:voc-v-update-full-all}
		\begin{align}  
				\label{eq:voc-v-update-full=midpoint}
				\vel[n+1] &= e^{\tau \mgFieldMatrix[]} \vel[n] + \tfrac{\tau}{2}
				\varphi_1(\tau \mgFieldMatrix[]) 
				\bigl(\elField[n] + \elField[n+1]\bigr).
			\intertext{Another possibly second-order method is obtained by using an exponential trapezoidal rule to approximate the integral, cf.\ \cite[Ex.~2.6]{HocO10}}
			\label{eq:voc-v-update-full=trap}
			\vel[n+1] &= e^{\tau \mgFieldMatrix[]} \vel[n] 
			+  \tau \bigl( \bigl(      \varphi_1(\tau \mgFieldMatrix[]) -  \varphi_2(\tau \mgFieldMatrix[]) \bigr)\elField[n]
			+   \varphi_2(\tau \mgFieldMatrix[]) \elField[n+1] \bigr).
		\end{align}
	\end{subequations}
	If we do not want to work on a staggered time grid for the velocities as in \eqref{eq:filt-boris-v}, we can combine \eqref{eq:voc-q-update} with one of the following variants
	\begin{subequations} 	\label{eq:voc-vhalf-update-all}
		\begin{align}
			\vel[n+1/2] & = \tfrac12 \bigl( e^{\tfrac{\tau}{2} \mgFieldMatrix[]} \vel[n]
			+ e^{-\tfrac{\tau}{2} \mgFieldMatrix[]} \vel[n+1] \bigr),
			\label{eq:voc-vhalf-update-vdk}\\
			\sinch(\tfrac{\tau}{2} \mgFieldMatrix[])\vel[n+1/2] & = \tfrac12  \bigl( \varphi_1(\tfrac{\tau}{2} \mgFieldMatrix[]) \vel[n]
			+ \varphi_1(-\tfrac{\tau}{2} \mgFieldMatrix[]) \vel[n+1] \bigr),
			\label{eq:voc-vhalf-update-vdk-neu}\\
			\vel[n+1/2] & =  e^{\tfrac{\tau}{2} \mgFieldMatrix[]} \vel[n] + \tfrac{\tau}{2}
			\varphi_1(\tfrac{\tau}{2} \mgFieldMatrix[]) \elField[n]. 
			\label{eq:voc-vhalf-update-vexpeul}
		\end{align}
	\end{subequations}
	
	The scheme \eqref{eq:voc-q-update} combined with the variants \eqref{eq:voc-v-update-full=midpoint}, \eqref{eq:voc-v-update-full=trap}, 
	\eqref{eq:voc-vhalf-update-vdk-neu},
	and \eqref{eq:voc-vhalf-update-vexpeul}
	are exact for constant magnetic and electric fields $\mgField[]$ and $\elField$, respectively.
	Moreover, \eqref{eq:id-phi1-phi2} shows that \eqref{eq:voc-v-update-full=HLW} is exact if $\elField$ is constant.
	However, the others fail to be exact in this situation, because $v_{n+1/2}$ is replaced by an approximation which does not have this property. For \eqref{eq:voc-vhalf-update-vdk}, exactness is only satisfied if $\elField = 0$.

	\section{Error analysis: recursion and stability} \label{sec:error}

	Our aim is an analysis of all different variants of filtered Boris-type algorithms in a uniform way and to characterize filter functions which lead to beneficial properties. Thus, we consider a generic exponential integrator \eqref{eq:gen-expint}. We are interested in the situation that
	\begin{equation}  \label{eq:tauB-lower-bound}
		\tau \normempty{\mgField} \geq \vartheta_0 
	\end{equation}
	for some given $\vartheta_0 \in (0,1)$. Otherwise, a standard, nonstiff error analysis yields order two errors for the position and velocity errors of the Boris algorithm without filters.
	
	For the method \eqref{eq:HLW-all} and some of our results, a nonresonance condition is required, so that one cannot take arbitrary stepsizes $\tau$.
	\begin{assumption}[Nonresonance condition] \label{ass:nonres}
		Let $\nonres \in (0,\pi)$ be fixed. The time stepsize $\tau$ satisfies the nonresonance condition 
		\begin{equation} \label{eq:nonres}
			{\Bigl|}\tau \normempty{\mgField[]} - 2 \pi k{\Bigr|} \geq \nonres \quad \text{for all} \quad k \in \mathbb{Z}
			\setminus\{0\}.
		\end{equation}
	\end{assumption}

	The notation $\mathcal{O}(\tau^{j})$ always means that we consider $\tau\to 0$ with constants depending on $T$, $K_{\pos}, K_{\vel}$ and bounds on $\elField$ and its derivatives in a neighborhood of $K_{\pos}$, but independent of $\tau$ and $\normempty{\mgField[]}$.

	\subsection{Definition of defects and error recursion} \label{subsec:error-recur}

	In the following, we use the short notation for the exact solution
	\begin{align*}
		\telField (t)& = \elField(\pos(t)),
		&
		\telFieldJac{1}(t) & = \tfrac{d}{dt}\telField(t) , 
		&
		\telFieldJac[n]{1} &= \tfrac{d}{dt}\elField(\tpos[n]\bigr),
		& \bdelFieldJac[n]{1} & = \max_{0 \leq t \leq \tn[n]} \normempty{\telFieldJac[]{1}(t)} \\
		\tpos[n] &= \pos(t_n),
		&
		\tvel[n] &= \vel(t_{n}),
		&
		\telField[n] &= \elField(\tpos[n]). &&
	\end{align*}

	We start by defining defects. Inserting the exact solution \eqref{eq:exact-q-v-voc} of \eqref{eq:eul-lag-filt-all} into the numerical scheme \eqref{eq:gen-expint} yields 
	defects $\defectPos[n+1]$ and $\defectVel[n+1]$ given by
	\begin{subequations}  \label{eq:gen-expint-defects}
		\begin{align}  
			\label{eq:gen-expint-defects-b}
			\tpos[n+1] &= \tpos[n] 
			+ \tfrac{\tau}{2} \bigl( \psifilter_- \tvel[n]+  \psifilter_+ \tvel[n+1] \bigr)
			+ \tfrac{\tau^2}{4} \bigl(
			\chifilter_-\telField[n]
			+  \chifilter_+ \telField[n+1] \bigr) + \tau \defectPos[n+1],\\
			\tvel[n+1] &= e^{\tau \mgFieldMatrix[]} \tvel[n] + \tfrac{\tau}{2}
			\bigl( \phifilter_- \telField[n]  + \phifilter_+ \telField[n+1] \bigr)
			+ \tau \defectVel[n+1].
			\label{eq:gen-expint-defects-a}
		\end{align}
	\end{subequations}
	We denote errors and defects by
	\begin{equation} \label{eq:errors-combined}
		\error[n] = 
		\begin{pmatrix}
			\errorPos[n]\\
			\errorVel[n]
		\end{pmatrix}
		= 
		\begin{pmatrix}
			\errorPos*[n]\\
			\errorVel*[n]
		\end{pmatrix}
		,
		\quad
		\defect[n] =
		\begin{pmatrix}
			\defectPos[n]\\
			\defectVel[n]
		\end{pmatrix} ,
		\quad
		\errorelFieldVec[n] =
		\begin{pmatrix}
			\errorelField[n]	\\
			\errorelField[n+1]
		\end{pmatrix},
		\quad
		\errorelField[n] = \errorelField*[n].
	\end{equation}
	Subtracting \eqref{eq:gen-expint} from \eqref{eq:gen-expint-defects} we obtain
	the error recursion 
	\begin{subequations} \label{eq:error-recur-all}
		\begin{equation} \label{eq:error-recur}
			\systemMatrix{-}
			\error[n+1] = 
			\systemMatrix{+}
			\error[n]
			+ 
			\tau
			\auxMatrix 
			\errorelFieldVec[n]
			+
			\tau \defect[n+1]
		\end{equation}
		with
		\begin{equation} \label{eq:error-recur-mats}
			\systemMatrix{-} = 
			\begin{pmatrix}
				\Id & -\tfrac{\tau}{2}  \psifilter_+\\
				0 & \Id
			\end{pmatrix},
			\qquad
			\systemMatrix{+} = 
			\begin{pmatrix}
				\Id & \tfrac{\tau}{2}  \psifilter_-\\
				0 & e^{\tau \mgFieldMatrix[]}
			\end{pmatrix},
			\qquad
			\auxMatrix =
			\tfrac12
			\begin{pmatrix}
				\frac{\tau}{2} \chifilter_{-} & \tfrac{\tau}{2} \chifilter_+\\
				\phifilter_{-} & \phifilter_{+}
			\end{pmatrix}.
		\end{equation}
	\end{subequations}
	Solving this recursion and writing $\systemMatrix{} = \systemMatrix{-}^{-1} \systemMatrix{+}$ yields
	\begin{equation} \label{eq:voc-discrete}
		\error[n] = \systemMatrix{}^n \error[0]
		+ \tau \sum_{j=0}^{n-1}  \systemMatrix{}^{n-j-1} \systemMatrix{-}^{-1}
		\auxMatrix \errorelFieldVec[j] + \defecterr[n]{}, \qquad
		\defecterr[n]{} = \tau \sum_{j=0}^{n-1}  \systemMatrix{}^{n-j-1} \systemMatrix{-}^{-1} \defect[j+1].
	\end{equation}
	\begin{lemma}  \label{lem:systemmatrix-power}
		The matrix $\systemMatrix{} = \systemMatrix{-}^{-1} \systemMatrix{+}$ satisfies
		\begin{subequations}
			\begin{equation} \label{eq:systemmat}
				\systemMatrix[n]{} =
				\begin{pmatrix}
					\Id &   \tau \upsilonfilter_n
					\\
					0 & e^{n\tau \mgFieldMatrix[]}
				\end{pmatrix}	,
				\qquad
				\systemMatrix[n]{}\systemMatrix{-}^{-1}
				=
				\begin{pmatrix}
					\Id & \tau(\tfrac12 \psifilter_+ +\upsilonfilter_n  )\\
					0 & e^{n\tau \mgFieldMatrix[]}
				\end{pmatrix}
				\qquad
				n=1,2,\ldots,
			\end{equation}
			where 
			\begin{equation} \label{eq:upsilonfilter}
				\upsilonfilter_n = \Upsilon_n(\tau \mgFieldMatrix[]) ,
				\quad
				\Upsilon_n(z)= \tfrac{1}{2}( \psi_-(z) + \psi_+(z)e^z )\funexpGauss_n(z)  , \quad
				\funexpGauss_n(z) = \sum_{j=0}^{n-1} e^{jz} = \frac{e^{nz}-1}{e^z-1}.
			\end{equation}
		\end{subequations}
		
	\end{lemma}
	\begin{proof}
		The block triangular structure of $\systemMatrix{}$ implies
		\begin{equation*}
			\systemMatrix{} =
			\begin{pmatrix}
				\Id &  \tau \Upsilon_1(\tau \mgFieldMatrix[])\\
				0 & e^{\tau \mgFieldMatrix[]}
			\end{pmatrix},
			\qquad
			\systemMatrix[n]{} = 
			\begin{pmatrix}
				\Id &  \tau \Upsilon_1(\tau \mgFieldMatrix[])
				\sum\limits_{j=0}^{n-1} e^{j\tau \mgFieldMatrix[]}\\
				0 & e^{n\tau \mgFieldMatrix[]}
			\end{pmatrix}.
		\end{equation*}
		As $\Upsilon_n(z)=\Upsilon_1(z)\funexpGauss_n(z)$, the
		claim follows from the definition of $\funexpGauss_n$ and $\funexpGauss_1 = 1$.
	\end{proof}
	
	\subsection{Stability} \label{subsec:stability}

	In order to study the stability, we investigate the function $\Upsilon_n$ defined in \eqref{eq:upsilonfilter} for the different choices of $\psi_{\pm}$ from Section~\ref{sec:exp-int}.
	Since $\Upsilon_n = \Upsilon_1 \funexpGauss_n$, we start with $\Upsilon_1$ and define the auxiliary function 
	\begin{equation}  \label{eq:Psi0-def}
		\Psi_0(z) 
		= \sinch(\tfrac{z}{2}) -  \Upsilon_1(z)  e^{ -\tfrac{z}{2}}.
	\end{equation}
	
	\begin{lemma}  \label{lem:psi-psi+exp}
		Let $\Upsilon_n$ be defined in \eqref{eq:upsilonfilter} and $\Psi_0$ in \eqref{eq:Psi0-def}, respectively. 
		If 	$\Psi_0= 0$, e.g., 
		for the filter functions of the schemes
			collected in Table~$\ref{tab:psichipmPsi01}$, we have $\Upsilon_1(z)= \varphi_1(z)$.
		\end{lemma}
		\begin{proof}
			Using \eqref{eq:id-sinch-exp}, we have 
			\begin{align*}
				\Psi_0(z) 
				= \bigl( \varphi_1(z) - \Upsilon_1(z) \bigr) e^{-z/2}.
			\end{align*}
			Hence $\Psi_0= 0 $ is equivalent to $\Upsilon_1= \varphi_1$ and the claim follows for \eqref{eq:HLW-all}, 
			\eqref{eq:voc-vhalf-update-vdk}, \eqref{eq:voc-vhalf-update-vdk-neu}, and 
			\eqref{eq:voc-vhalf-update-vexpeul}.
		\end{proof}
		
		\begin{theorem} \label{thm:power-bound}
			For the filter functions of the schemes in Table~$\ref{tab:psichipmPsi01}$,
			the power bounds $\normempty{\systemMatrix[n]{}} \leq 1+ \tn[n]$ and $\normempty{\systemMatrix[n]{}\systemMatrix{-}^{-1}} \leq 1+ \tn[n+1/2]$ hold for $n=0,1,2,\ldots$. Moreover, for all variants in \eqref{eq:voc-vhalf-update-all}
			we have $\normempty{\auxMatrix} \leq C_{\auxMatrix}$. For $\eqref{eq:HLW-all}$, this bound only holds if \eqref{eq:nonres} is satisfied.
			\end{theorem}
			\begin{proof}
				For an arbitrary matrix $\upsilonfilter$ we have the bound
				\begin{equation}  \label{eq:2x2-norm}
					\normempty{ 	
						\begin{pmatrix}
							\Id &  \tau \upsilonfilter \\
							0 & e^{n\tau \mgFieldMatrix[]}
						\end{pmatrix}
					}
					\leq 1 + \tau \normempty{\upsilonfilter}.
				\end{equation}
				Note that by \eqref{eq:upsilonfilter} and Lemma~\ref{lem:psi-psi+exp}, we can write
				\begin{equation} \label{eq:funexpGauss-phi1}
					\tau \Upsilon_n(z) =
					\tau \Upsilon_1 (z)\funexpGauss_n(z)  
					= \tau \varphi_1 (z)\frac{e^{nz}-1}{e^z-1}
					= t_n  \varphi_1(nz),
				\end{equation}
				so that
				\begin{equation} \label{eq:upsfilterbound}
					\normempty{\tau \upsilonfilter_n} = 	\max_{z \in i\mathbb{R}}\;\abs{\tau \Upsilon_n(z) }  \leq  t_n.
				\end{equation}
				With \eqref{eq:2x2-norm}, this proves $\normempty{\systemMatrix[n]{}} \leq 1+ \tn[n]$. The bound on $\normempty{\systemMatrix[n]{}\systemMatrix{-}^{-1}}$ follows from $\normempty{\psifilter_+}\leq 1$ and the bound on $\normempty{\auxMatrix}$ follows directly from $\mgFieldMatrix[]$ being skew-symmetric.
			\end{proof}

			Since $\mgFieldMatrix[]$ is skew-symmetric, there exists a unitary matrix $U\in \mathbb{C}^{3\times 3}$ of eigenvectors such that
			\begin{align}  \label{eq:mgmatrix-diag}
				\mgFieldMatrix[] = U\begin{pmatrix}
					\mgFieldMatrix[\perp] & \\  & 0
				\end{pmatrix}U^*, \qquad \mgFieldMatrix[\perp] = \begin{pmatrix}
					i \normempty{\mgField[]} & 0\\ 0 & -i \normempty{\mgField[]}
				\end{pmatrix}, \qquad U = \begin{pmatrix}
					U_\perp & U_\|
				\end{pmatrix},
			\end{align}
			with $U_\| = \mgField[]/\normempty{\mgField[]}$. The projector 
			onto the kernel of $\mgFieldMatrix[]$, i.e., the projection onto $\mathrm{span}(\mgField[])$, will be denoted by $\propar=U_\parallel U_\parallel^*$ and the corresponding orthogonal projection  by $\properp := \Id - \propar$, respectively. 
			Then, for a function $\phi$ defined on the spectrum of $\tau \mgFieldMatrix[]$, it holds
			\begin{equation} \label{eq:matfun-project}
				\phi(\tau \mgFieldMatrix[]) \properp = \properp\phi(\tau \mgFieldMatrix[])  = U_\perp \phi(\mgFieldMatrix[\perp]) U_\perp^*.
			\end{equation}
			
			If the nonresonance condition \eqref{eq:nonres} is satisfied, we can refine the stability bound. 
			\begin{lemma}  \label{lem:rhon-bound}
				If $\tau$ is chosen such that the nonresonance Assumption~$\ref{ass:nonres}$ holds for some $\vartheta \in (0,\pi)$ and \eqref{eq:tauB-lower-bound} is satisfied, then   $\funexpGauss_n$ defined in \eqref{eq:upsilonfilter} satisfies
				\begin{equation} \label{eq:funexpgauss-bound}
					\normempty{  \properp \funexpGauss_n(\tau\mgFieldMatrix[])} 
					= \normempty{   \funexpGauss_n(\tau\mgFieldMatrix[])\properp} 
					\leq \normempty{   \funexpGauss_n(\tau\mgFieldMatrix[\perp])} 
					\leq \frac{1}{  \abs{\sin \tfrac{\min\{\vartheta,\vartheta_0\}}{2}}}, 
					\qquad
					n \in \mathbb{N}		.
				\end{equation}
			\end{lemma}
			\begin{proof}
				For $x = \pm\tau\normempty{\mgField[]}$ it holds
				\begin{equation*}
					\abs{ (e^{ix} - 1)^{-1} } = \frac{1}{ 2\abs{\sin(\tfrac12 x)}}
					\leq \frac{1}{2  \abs{\sin \tfrac{\min\{\vartheta,\vartheta_0\}}{2}}}.
				\end{equation*}
				The bound \eqref{eq:funexpgauss-bound} follows from the definition of $\funexpGauss_n$.
			\end{proof}

			\section{Error bounds of order one} \label{sec:order-1}
			
			In this section, we characterize filter functions in the general scheme \eqref{eq:gen-expint} which lead to error bounds of order one with constants independent of $\normempty{\mgField}$.
			Our analysis requires the following assumption in the whole section without further mention.
			\begin{assumption}\label{ass:efield-assump-C1}
				The trajectories of $\pos, \vel$ remain in compact sets $K_{\pos}, K_{\vel}$, independent of $\normempty{\mgField[]}$, and $\elField \in C^1$ in a neighborhood of $K_{\pos}$.
			\end{assumption}

			We start with deriving explicit expressions for the dominant term in the defects in \eqref{eq:gen-expint-defects}.
			\begin{lemma} \label{lem:defects-q}
				The defects \eqref{eq:gen-expint-defects-b} for the positions satisfy
				\begin{subequations} 	\label{eq:defectPos-all}
					\begin{align} 	\label{eq:defectPos}
						\defectPos[n+1] = \Psi_0(\tau\mgFieldMatrix[]) \tvel[n+1/2] 
						+ \tau \Psi_1(\tau\mgFieldMatrix[]) \telField[n+1/2] + \mathcal{O}(\tau^2),
					\end{align}
					with $\Psi_0$ defined in \eqref{eq:Psi0-def} and 
					\begin{align}\label{eq:Psi1-def}
						\Psi_1(z)  &=  \tfrac12 \psi_{1,{\sinh}}(\tfrac{z}{2}) 
						- \tfrac14 \bigl( \psi_{+}(z) \varphi_1(\tfrac{z}{2})
						-  \psi_{-}(z) \varphi_1(-\tfrac{z}{2})\bigr)
						-	\tfrac14 \bigl( \chi_-(z)+  \chi_+(z)\bigr),
					\end{align}
					where $\psi_{1,{\sinh}}(z)= (\sinh(z)-z)/z^2$. 
				\end{subequations}
			\end{lemma}
			\begin{proof}
				The proof relies on expansion of $\telField[]$ at $t_{n+1/2}$ 
				\begin{equation}\label{eq:expand-elField}
					\begin{aligned}
						\telField(t_{n+1/2} \pm \tfrac{\tau}{2} \sigma) &= \telField[n+1/2] +  \tfrac{\tau}{2} \remainelField[n+1/2]{1}(\pm\sigma),\\
						\remainelField[n+1/2]{1}(\sigma) &= 
						\int_{0}^{\sigma}   \telFieldJac{1}(\tn[n+1/2] + \tfrac{\tau}{2} \theta) d\theta
					\end{aligned}
				\end{equation}
				within the variation-of-constants formula \eqref{eq:exact-q-v-voc-a}. With $t=t_{n+1/2}$ and $\theta = \tau/2$, \eqref{eq:exact-q-v-voc-a} can be written as
				\begin{subequations}\label{eq:posnp1-exact-all}
					\begin{align}
						\label{eq:posnp1-exact-a}
						\tpos[n+1] & = \tpos[n+1/2] 
						+ \tfrac{\tau}{2} \varphi_1\bigl(\tfrac{\tau}{2} \mgFieldMatrix[]\bigr) \tvel[n+1/2]
						\\&\quad+\tfrac{\tau^2}{4} \int_0^1 (1-\sigma)
						\varphi_1\bigl( \tfrac{\tau}{2}(1-\sigma)\mgFieldMatrix[]\bigr)
						\telField(t_{n+1/2}+\tfrac{\tau}{2}\sigma)
						d\sigma,\nonumber\\	
						\label{eq:posnp1-exact-b}
						\tpos[n] & = \tpos[n+1/2] 
						- \tfrac{\tau}{2} \varphi_1\bigl(-\tfrac{\tau}{2} \mgFieldMatrix[]\bigr) \tvel[n+1/2]
						\\&\quad+\tfrac{\tau^2}{4}\int_0^1 (1-\sigma)
						\varphi_1\bigl(-\tfrac{\tau}{2} (1-\sigma)\mgFieldMatrix[]\bigr)
						\telField(t_{n+1/2}-\tfrac{\tau}{2}\sigma)
						d\sigma.\nonumber	
					\end{align}
				\end{subequations}
				Subtracting \eqref{eq:posnp1-exact-b} from \eqref{eq:posnp1-exact-a} and using 
				\begin{equation*}
					\int_0^1 (1-\sigma)
					\varphi_1\bigl(\pm\tfrac{\tau}{2} (1-\sigma)z\bigr) 	
					d\sigma = 
					\varphi_2 (\pm\tfrac{\tau}{2}z\bigr),  
				\end{equation*}
				we obtain with \eqref{eq:id-sinch}  and \eqref{eq:id-psi1sinh}
				\begin{align} \label{eq:posnp1-exact}
						\tpos[n+1]  &= \tpos[n] 
						+ \tau \sinch \bigl(\tfrac{\tau}{2} \mgFieldMatrix[]\bigr) \tvel[n+1/2]
						\\&\quad+\tfrac{\tau^2}{4} \int_0^1 (1-\sigma)
						\varphi_1\bigl( \tfrac{\tau}{2}(1-\sigma)\mgFieldMatrix[]\bigr)
						\telField(t_{n+1/2}+\tfrac{\tau}{2}\sigma)
						d\sigma
						\nonumber
						\\&\quad-\tfrac{\tau^2}{4}\int_0^1 (1-\sigma)
						\varphi_1\bigl(-\tfrac{\tau}{2} (1-\sigma)\mgFieldMatrix[]\bigr)
						\telField(t_{n+1/2}-\tfrac{\tau}{2}\sigma)
						d\sigma
						\nonumber
						\\
						& = \tpos[n] 
						+ \tau \sinch \bigl(\tfrac{\tau}{2} \mgFieldMatrix[]\bigr) \tvel[n+1/2]
						+ \tfrac{\tau^2}{2} \psi_{1,\sinh}(\tfrac{\tau}{2} \mgFieldMatrix[])
						\telField[n+1/2]
						+ \tau \defectPos[n+1]',\nonumber
				\end{align}
				where
				\begin{align*}
					\defectPos[n+1]' &= 
					\tfrac{\tau^2}{8} \int_0^1 (1-\sigma)
					\varphi_1\bigl( \tfrac{\tau}{2}(1-\sigma)\mgFieldMatrix[]\bigr)
					\remainelField[n+1/2]{1}(\sigma)d\sigma \\&\quad
					-\tfrac{\tau^2}{8} \int_0^1 (1-\sigma)\varphi_1\bigl(-\tfrac{\tau}{2} (1-\sigma)\mgFieldMatrix[]\bigr)
					\remainelField[n+1/2]{1}(-\sigma)
					d\sigma	.
				\end{align*}
				From 
				\begin{equation*}
					\pm \tfrac18 \int_0^1 (1-\sigma) \sigma 
					\varphi_1\bigl( \pm \tfrac{\tau}{2}(1-\sigma)z\bigr) d\sigma
					= \varphi_3 \bigl(\pm \tfrac{\tau}{2}z\bigr)
				\end{equation*}
				and $\abs{\varphi_3(ix)} \leq 1/6$ for $x \in \mathbb{R}$, we conclude 
				\begin{equation*}
					\normempty{\defectPos[n+1]'} \leq 
					\tfrac{\tau^2}3  \max_{0 \leq t \leq \tn[n]} \normempty{\telFieldJac[]{1}(t)}
					=\tfrac{\tau^2}3  \bdelFieldJac[n+1]{1} 
					.
				\end{equation*}
				Using \eqref{eq:exact-q-v-voc-b} and \eqref{eq:expand-elField}, we obtain
				\begin{align*}
					\psifilter_-\tvel[n]+ 	\psifilter_+\tvel[n+1] 
					&= \bigl( \psifilter_{-} e^{-\tfrac{\tau}{2} \mgFieldMatrix[] }
					+  \psifilter_{+} e^{\tfrac{\tau}{2} \mgFieldMatrix[] }\bigr) \tvel[n+1/2]\\
					& 
					+  \tfrac{\tau}{2}\bigl( \psifilter_{+} \varphi_1(\tfrac{\tau}{2} \mgFieldMatrix[] )
					-  \psifilter_{-} \varphi_1(-\tfrac{\tau}{2} \mgFieldMatrix[] )\bigr)\telField[n+1/2]
					+ \defectPos[n+1]'',
				\end{align*}
				with
				\begin{align*}
					\defectPos[n+1]''
					= \tfrac{\tau^2}{4} \int_0^1\Bigl(\psifilter_{+} e^{\tfrac{\tau}{2}(1-\sigma)\mgFieldMatrix[]}\remainelField[n+1/2]{1}(\sigma)
					+ \psifilter_{-} e^{-\tfrac{\tau}{2}(1-\sigma)\mgFieldMatrix[]}\remainelField[n+1/2]{1}(-\sigma)  \Bigr)d\sigma.
				\end{align*}
				Hence, we conclude for $\defectPos[n+1]$ defined by \eqref{eq:gen-expint-defects-b} with the same calculations as above for the corresponding expression involving $\elField$
				\begin{align*}
					\defectPos[n+1]  
					=&  \sinch \bigl(\tfrac{\tau}{2} \mgFieldMatrix[]\bigr) \tvel[n+1/2] 
					- \tfrac12 \bigl( 	\psifilter_-\tvel[n]+ 	\psifilter_+\tvel[n+1] \bigr)
					\nonumber
					\\
					& 
					+\tfrac{\tau}{2} \Bigl(	\psi_{1,{\sinh}}(\tfrac{\tau}{2} \mgFieldMatrix[]) 	\telField[n+1/2] 
					-  	\tfrac12 \bigl( \chifilter_-\telField[n]
					+  \chifilter_+ \telField[n+1] \bigr) \Bigr)
					+ \defectPos[n+1]'
					\nonumber
					\\
					=& \Bigl(\!\sinch \bigl(\tfrac{\tau}{2} \mgFieldMatrix[]\bigr)
					\!-\!  	\tfrac12 \bigl( \psifilter_{-} e^{-\tfrac{\tau}{2} \mgFieldMatrix[] }
					\!+\!  \psifilter_{+} e^{\tfrac{\tau}{2} \mgFieldMatrix[] }\bigr)\!\Bigr)	\tvel[n+1/2]\\
					&	+  \!\tfrac{\tau}{2} \Bigl( \psi_{1,{\sinh}}(\tfrac{\tau}{2} \mgFieldMatrix[])   
					\!-\! \tfrac12 \bigl(\!\psifilter_{+} \varphi_1(\tfrac{\tau}{2} \mgFieldMatrix[] )
					\!-\!  \psifilter_{-} \varphi_1(-\tfrac{\tau}{2} \mgFieldMatrix[] )\bigr)
					\!-\!	\tfrac12 \bigl( \chifilter_- \!+\!  \chifilter_+\bigr)\!\Bigr) \! \telField[n+1/2]
					\\& + \defectPos[n+1]'
				\end{align*}
				where 
				\begin{equation*}
					\defectPos*[n+1] = \defectPos[n+1]' - \defectPos[n+1]'',
					\qquad \normempty{\defectPos*[n+1]} 
					\leq \tau^2 \Bigl(\tfrac13 + \tfrac14\normempty{\psifilter_{+}}  + \tfrac14\normempty{\psifilter_{-}}\Bigr)
					\bdelFieldJac[n+1]{1}.
				\end{equation*}
				This proves the claim.
			\end{proof}
			
			\begin{lemma} \label{lem:defects-v}
				The defects $\defectVel[n+1]$ for the velocities defined in \eqref{eq:gen-expint-defects-a} satisfy
				\begin{subequations} 	\label{eq:defectVel-all}
					\begin{align} 	\label{eq:defectVel-first}
						\defectVel[n+1] &= \Phi_0(\tau\mgFieldMatrix[])   \telField[n+1/2] +  \defectVel*[n+1], 
						\qquad \Phi_0 = \varphi_1 - \tfrac12 (\phi_+ +\phi_- ),
					\end{align}
					and
					\begin{equation} \label{eq:tdefectVel-bound}
						\normempty{\defectVel*[n+1]} 
						\leq  
						\tfrac{\tau}4\bigl(2 + \normempty{\phifilter_{-}} + \normempty{\phifilter_{+}} \bigr)\bdelFieldJac[n+1]{1}.
					\end{equation}
				\end{subequations}
			\end{lemma}
			\begin{proof}
				By \eqref{eq:exact-q-v-voc-b} with $t=t_n$ and $\theta = \tau$ and  \eqref{eq:expand-elField}, we obtain
				\begin{subequations} \label{eq:defectVel-1-all}
					\begin{equation} \label{eq:defectVel-phis}
						\begin{aligned}
							\defectVel[n+1] &=  \int_0^1 e^{\tau(1-\sigma)\mgFieldMatrix[]} 
							\telField(t_n+\tau\sigma) 
							d\sigma 
							- \tfrac12\bigl( \phifilter_- \telField[n]  + \phifilter_+ \telField[n+1] \bigr)
							\\
							& = \varphi_1(\tau\mgFieldMatrix[])\telField[n+1/2] 
							- 	\tfrac12(\phifilter_-  + \phifilter_+) \telField[n+1/2]
							+ \defectVel*[n+1],
						\end{aligned}
					\end{equation}
					with 
					\begin{align} \label{eq:defectVel-tvel}
						\defectVel*[n+1] & = 
						\tfrac{\tau}{2}\!\int_0^1 e^{\tau(1-\sigma)\mgFieldMatrix[]} \remainelField[n+1/2]{1}(2\sigma-1) d\sigma
						\!-\! \tfrac{\tau}4\bigl( \phifilter_-\remainelField[n+1/2]{1}(-1) + \phifilter_{+} \remainelField[n+1/2]{1}(1) \bigr).
					\end{align}
				\end{subequations}
				
				For the remainder we have with \eqref{eq:expand-elField}
				\begin{equation*}
					\normempty{ \remainelField[n+1/2]{1}(\sigma)}
					\leq  \abs{\sigma} \bdelFieldJac[n+1/2]{1}
					\leq  \bdelFieldJac[n+1/2]{1},
					\qquad
					-1 \leq \sigma \leq 1.
				\end{equation*}
				This proves \eqref{eq:defectVel-all}.
			\end{proof}
			
			We collect the formulas for the filters and the defect matrices for the filtered Boris variant \eqref{eq:HLW-all} and the update formulas \eqref{eq:voc-v-update-full-all} in the Tables~\ref{tab:phipmPhi01} and \ref{tab:psichipmPsi01}, starting with the velocities. 
			The formulas for \eqref{eq:voc-v-update-full=HLW} follow from the identities \eqref{eq:id-phi1-phi2} and \eqref{eq:id-phi3diff} and Lemma~\ref{lem:HLW-Psi_1}.
			Note that all variants satisfy $\Psi_1(0)=0$ which means that the position defects \eqref{eq:defectPos} are second order in $\tau$, but possibly with a constant depending on $\normempty{\mgField[]}$. 
			Only for the variants with $\Psi_1 = 0$, the defects are in 
			$\mathcal{O}(\tau^2)$ with a constant independent of $\normempty{\mgField[]}$.
			
			\begin{table}
				\caption{Overview on functions $\phi_\pm$ and $\Phi_{0,1}$ for the different updates of the velocity.}
					$
					\renewcommand{\arraystretch}{1.5}
					\begin{array}{c|c|c|c|c|c}
						& \phi_-(z)& \phi_+(z) & \Phi_0(z) & \Phi_1(z) & \Phi_\perp(z) \\
						\hline
						\text{\eqref{eq:voc-v-update-full=midpoint}}        &\varphi_1(z) &\varphi_1(z) & 0 & \varphi_2(z) - \tfrac12 \varphi_1(z)
						& \tfrac1{z} ( 1-\sinch (z) )\\
						\text{\eqref{eq:voc-v-update-full=trap}} & 2 \bigl(\varphi_1(z) -  \varphi_2(z)\bigr) & 2 \varphi_2(z)
						& 0 & 0 & \tfrac1{z} ( 1-\sinch^2 (\tfrac{z}{2}) ) \\
						\text{\eqref{eq:voc-v-update-full=HLW}} & 2 \frac{\varphi_2(z)}{\varphi_1(-z)} & 2 \frac{\varphi_2(- z)}{\varphi_1(-z)} & 0 & \frac{\varphi_3(z) - \varphi_3(-z)}{\varphi_1(-z)}
						& 0
					\end{array}
					$
				\label{tab:phipmPhi01}
			\end{table}
			
			\begin{table}
				\caption{Overview on functions $\psi_\pm$, $\chi_\pm$ and $\Psi_{1}$ for the different updates of the position. For all the choices of filter functions in the table, we have $\Psi_{0}=0$.}
					$
					\renewcommand{\arraystretch}{1.5}
					\begin{array}{c|c|c|c|c|c}
						& \psi_-(z)& \psi_+(z) & \chi_-(z) & \chi_+(z) & \Psi_1(z) \\
						\hline
						\text{\eqref{eq:voc-vhalf-update-vdk}}        &  \varphi_1(z) &  \varphi_1(-z) & -\varphi_2(\tfrac{-z}{2}) & \varphi_2( \tfrac{z}{2}) &
						\tfrac{z}{4}\sinch(\tfrac{z}{2})\psi_{1,{\cosh}}(\tfrac{z}{2}) 		
						\\
						\text{\eqref{eq:voc-vhalf-update-vdk-neu}}     &  \varphi_1( \tfrac{z}{2}) &  \varphi_1(\tfrac{-z}{2}) & -\varphi_2( \tfrac{-z}{2}) & \varphi_2( \tfrac{z}{2}) & 0 \\
						\text{\eqref{eq:voc-vhalf-update-vexpeul}} & 2\varphi_1(z) & 0 & 2\varphi_1(z)\varphi_1(\tfrac{z}{2})-\varphi_2(\tfrac{-z}{2})
						& \varphi_2( \tfrac{z}{2}) & -\tfrac{z}{2} \psi_{1,\cosh}(\tfrac{z}{2}) \varphi_1(z)\\
						\text{\eqref{eq:HLW_q}} & 2 \varphi_1(z) & 0 & 4 \varphi_2(z) & 0 & 0
					\end{array}
					$
				\label{tab:psichipmPsi01}
			\end{table}

			\begin{theorem} \label{thm:error-bound}
				Let $\Phi_0=\Psi_0=0$ and let $\lipschitzelField$ be the Lipschitz constant of $\elField$. Further assume $\tau$ to be sufficiently small s.t. $2\tau (1 + T)C_{\auxMatrix} \lipschitzelField < 1$ and that the filter functions are chosen such that $\normempty{\auxMatrix} \leq C_{\auxMatrix}$. Then, the error 
				of the approximation \eqref{eq:gen-expint} to the solution of \eqref{eq:eul-lag-filt-all} is bounded by
				\begin{equation*}
					\normempty{\error[n]} \leq C \tau, \qquad 0\leq \tn[n] = n\tau \leq T,
				\end{equation*}
				where $C = C(\lipschitzelField, K, \auxMatrix, T)$.
			\end{theorem}
			\begin{proof}
				From the definition of $\errorelFieldVec[n]$  in \eqref{eq:errors-combined} we have $\normempty{\errorelFieldVec[n]} \leq 
				\normempty{\errorelField[n]}+ \normempty{\errorelField[n+1]}$. Then, 
				\eqref{eq:voc-discrete}	and the power bound from Theorem~\ref{thm:power-bound} imply
				\begin{align*}
					\normempty{\error[n]} 
					&\leq \tau (1+T) \Bigl(\sum_{j=0}^{n-1} \normempty{\defect[j+1]}
					+ 2 C_{\auxMatrix} \sum_{j=0}^{n}	 \normempty{\errorelField[j]} \Bigr)\\
					&\leq \tau  (1+T) \Bigl(\sum_{j=0}^{n-1} \normempty{\defect[j+1]}
					+ 2 C_{\auxMatrix} \lipschitzelField \sum_{j=0}^{n}	 \normempty{\errorPos[j]} \Bigr).
				\end{align*}
				The claim now follows from Lemmas~\ref{lem:defects-q} and \ref{lem:defects-v} together with a standard Gronwall argument.
			\end{proof}
			
			\begin{corollary}  \label{cor:CS-bound}
				For the filter functions $\phi_\pm$ given in \eqref{eq:voc-v-update-full=midpoint} or \eqref{eq:voc-v-update-full=trap} and $\chi_\pm$ given by one of the variants in 
				\eqref{eq:voc-vhalf-update-all}, we have $C_{\auxMatrix}^2 \leq
				\tfrac12 (1+\tau^2)$.
			\end{corollary}
			\begin{proof}
				By definition \eqref{eq:error-recur-mats}, we have 
				\begin{equation*}
					C_{\auxMatrix}^2  \leq \tfrac14 \Bigl( \tfrac{\tau^2}{4}
					\bigl(\normempty{\chifilter_+}^2 + \normempty{\chifilter_-}^2 \bigr)
					+ \normempty{\phifilter_+}^2  + \normempty{\phifilter_-}^2 \Bigr).  
				\end{equation*}
				The statement now follows from $\normempty{\chifilter_+} \leq 1/2$, $\normempty{\chifilter_-} \leq 2, \normempty{\phifilter_\pm} \leq 1$ for all choices.
			\end{proof}

			\section{Error bounds of order two}  \label{sec:order-2}

			Unfortunately, the analysis of the previous section cannot be generalized in a straightforward way to second-order error bounds. The problem is that the second-order expansion of $\vel$ includes a second time-derivative of $\elField(\pos(t))$, which cannot be bounded independently of $\normempty{\mgField[]}$ because of \eqref{eq:eul-lag-filt-b}. Hence, to prove second-order error estimates for the position, we have to investigate the dominant error terms of $\vel$ and the error propagation more carefully.
			
			Our analysis requires a stronger assumption than Assumption~\ref{ass:efield-assump-C1}, which we assume in the following without further mention.
			\begin{assumption}\label{ass:efield-assump}
				The trajectories of $\pos, \vel$ remain in compact sets $K_{\pos}, K_{\vel}$, independent of $\normempty{\mgField[]}$, and $\elField \in C^2$ in a neighborhood of $K_{\pos}$.
			\end{assumption}

			\subsection{Refined representation of the velocity defect} \label{subsec:defects}
			
			In the following lemma we show that the velocity defect can be split into two parts which can later be handled differently in the error recursion.

			\begin{lemma}[Velocity defect]  \label{lem:defect-v-direct}
				Suppose that the filter functions in \eqref{eq:gen-expint} are chosen such that $\Phi_0 = 0$ holds in \eqref{eq:defectVel-first}. Then the velocity defect $\defectVel[n+1]$ can be written as 
				\begin{subequations} \label{eq:defect-v-direct-all}
					\begin{align}  \label{eq:defect-v-direct}
						\defectVel[n+1] &=\tau
						\Phi_1(\tau\mgFieldMatrix[]) \elField'(\tpos[n+1/2]) \rem_{n+1/2}
						+ \tau \defectVel[n+1]'
						+  \mathcal{O}(\tau^2),\\
						\intertext{where $\vel[\perp]  = \properp \vel$ and}
						\label{eq:phi1}
						\Phi_1 &= \varphi_2 - \tfrac12 \varphi_1 - \tfrac14 (	\phi_+ - \phi_-)
						= \varphi_2 - \tfrac12 \phi_+				,\\
						\label{eq:rem-def}
						\rem_{n+1/2} &= \tvel[n+1/2] - e^{\tn[n+1/2]\mgFieldMatrix[]}\vel[\perp](0),
						\\
						\defectVel[n+1]' & = \Bigl(\int_0^1  (\sigma-\tfrac12) e^{\tau(1-\sigma)\mgFieldMatrix[]}
						\elField'(\tpos[n+1/2])
						\varphi_1\bigl( (\sigma-\tfrac12) \tau \mgFieldMatrix[]\bigr)d\sigma
						\nonumber\\
						&\quad 
						- \tfrac{1}{4} \bigl(\phifilter_+  \elField'(\tpos[n+1/2])
						\varphi_1(\tfrac{\tau}{2}\mgFieldMatrix[])
						- \phifilter_-  \elField'(\tpos[n+1/2])
						\varphi_1(-\tfrac{\tau}{2}\mgFieldMatrix[])\bigr)
						\Bigr)
						e^{\tn[n+1/2]\mgFieldMatrix[]} \vel[\perp](0).
						\label{eq:defect-v-strich}
					\end{align}
				\end{subequations}
			\end{lemma}
			\begin{proof}
				We start from the representation \eqref{eq:defectVel-first}, where by the assumption that $\Phi_0=0$, $\defectVel[n+1] = \defectVel*[n+1]$ given by 	
				\eqref{eq:defectVel-tvel}. Since $\elField \in C^2$, the remainder defined in \eqref{eq:expand-elField}, where $\telFieldJac{1}(t) = \elField'(\pos(t)) \vel(t)$, can be written as
				\begin{align} 
					\remainelField[n+1/2]{1}(\sigma) &= \int_0^\sigma \elField'\bigl(\pos(\tn[n+1/2] + 	\tfrac{\tau}{2}\theta)\bigr)\vel\bigl(\tn[n+1/2] + \tfrac{\tau}{2}\theta\bigr) d\theta
					\nonumber\\
					& = \elField'(\tpos[n+1/2])\int_{0}^\sigma \vel(\tn[n+1/2] + \tfrac{\tau}{2}\theta) d\theta + \mathcal{O}(\tau)
					\nonumber \\
					& = \elField'(\tpos[n+1/2])
					\Bigl( 
					\sigma\tvel[n+1/2] \nonumber\\&\quad- \tfrac{\tau}{2}\int_{0}^\sigma (\theta-\sigma) \bigl(\mgFieldMatrix[]\vel(\tn[n+1/2] + \tfrac{\tau}{2}\theta) + \telField(\tn[n+1/2] + \tfrac{\tau}{2}\theta) \bigr) d\theta
					\Bigr)
					+ \mathcal{O}(\tau)
					\nonumber\\
					& = \elField'(\tpos[n+1/2])
					\Bigl( 
					\sigma\tvel[n+1/2] - \tfrac{\tau}{2}\int_{0}^\sigma (\theta-\sigma) 		\mgFieldMatrix[]\vel(\tn[n+1/2] + \tfrac{\tau}{2}\theta)d\theta
					\Bigr) 
					+ \mathcal{O}(\tau)
					\label{eq:remainder-2-all}
				\end{align}
				For the third identity, we used integration by parts 
				with $1=\tfrac{d}{d\theta}(\theta-\sigma)$ and \eqref{eq:eul-lag-filt-b}.
				
				Furthermore, as $\mgFieldMatrix[]\propar\vel(t) = 0$, it suffices to consider $\vel[\perp](t) := \properp\vel(t)$, which, by \eqref{eq:eul-lag-filt-b}, evolves as
				\begin{equation*}
					\dot{\vel}_\perp = \mgFieldMatrix[]\vel[\perp](t) +  \properp\telField(t).
				\end{equation*} 
				Using the variation-of-constants formula, we thus obtain 
				\begin{align*}
					\vel[\perp](t) &= 
					e^{t\mgFieldMatrix[]}\vel[\perp](0) 
					+
					\int_0^t e^{(t-s)\mgFieldMatrix[]} \properp\telField(s)ds.
					\intertext{Multiplying by $\mgFieldMatrix[]$ and integration by parts yields}
					\mgFieldMatrix[] \vel(t) =
					\mgFieldMatrix[] \vel[\perp](t) & = \mgFieldMatrix[] e^{t\mgFieldMatrix[]}\vel[\perp](0) 
					- \properp \telField(t)+ e^{t\mgFieldMatrix[]} \properp\telField(0)
					+ \int_0^t e^{(t-s)\mgFieldMatrix[]} \properp\telFieldJac{1}(s)ds\\
					& =  \mgFieldMatrix[] e^{t\mgFieldMatrix[]}\vel[\perp](0) + \mathcal{O}(1).
				\end{align*}
				Using again $\elField\in C^2$ and integration by parts we obtain
				\begin{align*}
					\tfrac{\tau}{2}\int_{0}^\sigma (\theta -\sigma)  \mgFieldMatrix[]\vel(\tn[n+1/2] + \tfrac{\tau}{2}\theta)  d\theta 
					&= \sigma e^{\tn[n+1/2]\mgFieldMatrix[]}\vel[\perp](0) 
					- \int_{0}^\sigma e^{\tfrac{\tau}{2}\theta\mgFieldMatrix[]} d\theta \, e^{\tn[n+1/2]\mgFieldMatrix[]}\vel[\perp](0) 
					\\&\quad+ \mathcal{O}(\tau).
				\end{align*}
				Inserting this into  \eqref{eq:remainder-2-all} and using
				\begin{equation} \label{eq:phi1-limits}
					\int_{0}^{\sigma} e^{\tfrac{\tau}{2}\theta \mgFieldMatrix[]} d\theta
					= \sigma \varphi_1\bigl( \sigma \tfrac{\tau}{2} \mgFieldMatrix[]\bigr)
				\end{equation}
				yields 
				\begin{align}  \label{eq:remain-v-1-direct}
					\remainelField[n+1/2]{1}(\sigma) = \sigma\elField'(\tpos[n+1/2])\bigl(\rem_{n+1/2} + \varphi_{1}(\sigma\tfrac{\tau}{2}  \mgFieldMatrix[])e^{\tn[n+1/2]\mgFieldMatrix[]}\vel[\perp](0)\bigr) + \mathcal{O}(\tau).
				\end{align}
				Finally, we insert \eqref{eq:remain-v-1-direct} into \eqref{eq:defectVel-tvel} and evaluate the expression in front of  $\rem_{n+1/2}$
				\begin{align*}
					\tfrac12 \int_0^1 (2\sigma-1) e^{\tau(1-\sigma)\mgFieldMatrix[]}d\sigma  
					- \tfrac14(\phifilter_+
					- \phifilter_-)&= \varphi_2(\tau\mgFieldMatrix[]) 
					- \tfrac12 \varphi_1(\tau\mgFieldMatrix[]) 
					- \tfrac14( \phifilter_+ -\phifilter_-) \\&= \Phi_1(\tau\mgFieldMatrix[]).
				\end{align*}
				Collecting the other terms in $\defectVel[n+1]'$ proves the claim.
			\end{proof}

			\subsection{Errors in parallel and perpendicular direction} \label{subsec:error-split}
			
			A crucial technique to analyze the error is to decompose it into contributions parallel and perpendicular to the magnetic field. Recall that the error in \eqref{eq:voc-discrete} splits into two terms: $\defecterr[n]{}$ resulting from the defects and a second part coming from the error in the evaluation of $\elField$. We start with $\defecterr[n]{}$.
			
			\begin{lemma}[Parallel error contribution]  \label{lem:error-par}
				Let Assumption~$\ref{ass:nonres}$ be satisfied. If the filter functions are chosen such that $\Phi_0(0) = \Phi_1(0) = \Psi_0(0) = \Psi_1(0) = 0$, then the $\propar$ projection of the defect contribution in \eqref{eq:voc-discrete} satisfies
				\begin{align*}
					\defecterr[\pos,n]{\|} = \begin{pmatrix}
						\propar & 0
					\end{pmatrix}\defecterr[n]{} = \mathcal{O}(\tau^2), \qquad \defecterr[\vel,n]{\|} = \begin{pmatrix}
						0 & \propar
					\end{pmatrix}\defecterr[n]{} = \mathcal{O}(\tau^2).
				\end{align*}
			\end{lemma}
			\begin{proof}
				We look at the position component first. Lemma~\ref{lem:systemmatrix-power}
				gives
				\begin{align}
					\defecterr[\pos,n]{\|} &= \propar \tau \sum_{j=0}^{n-1} \Bigl(\defectPos[j+1] + \bigl(\tfrac{\tau}{2}\psifilter_+ + \tau \upsilonfilter_{n-j-1}\bigr)\defectVel[j+1]\Bigr)
					\nonumber
					\\ 
					&= \tau \sum_{j=0}^{n-1} \bigl(\propar \defectPos[j+1] + (\tfrac{\tau}{2}\psi_+(0) + \tn[n-j-1])\propar \defectVel[j+1]\bigr).
					\label{eq:defectpos-par}
				\end{align}
				Using \eqref{eq:defectPos} and $\Psi_0(0) = \Psi_1(0) = 0$, we get $\propar \defectPos[j+1] = \mathcal{O}(\tau^2)$. For the velocity defect, we use the representation from Lemma~\ref{lem:defect-v-direct} and note that by \eqref{eq:defectVel-first} and \eqref{eq:phi1}, the assumptions $\Phi_0(0) = \Phi_1(0) = 0$ imply $\phi_\pm(0)=1$.
				This yields
				\begin{align} \label{eq:peak-order-func}
					\propar \defectVel[j+1]
					&= 	\propar \defectVel[j+1]' + \mathcal{O}(\tau^2) \nonumber \\ 
					&=\tau \propar\elField'(\tpos[j+1/2])\Bigl(\int_0^1  (\sigma-\tfrac12) 		
					\varphi_1\bigl( (\sigma-\tfrac12) \tau \mgFieldMatrix[]\bigr)d\sigma 
					\nonumber\\&\quad- \tfrac{1}{4} \bigl(
					\varphi_1(\tfrac{\tau}{2}\mgFieldMatrix[])
					- 
					\varphi_1(-\tfrac{\tau}{2}\mgFieldMatrix[])\bigr)
					\Bigr)
					e^{\tn[j+1/2]\mgFieldMatrix[]} \vel[\perp](0) +  \mathcal{O}(\tau^2),\nonumber
					\\
					&= \tau\propar\elField'(\tpos[j+1/2])  e^{\tn[j]\mgFieldMatrix[]} (\varphi_2(\tau\mgFieldMatrix[]) - \tfrac12\varphi_1(\tau\mgFieldMatrix[])) \vel[\perp](0) +  \mathcal{O}(\tau^2).
				\end{align}
				Since the projectors commute with the matrix functions, cf.\ \eqref{eq:matfun-project}, this yields
				\begin{equation}  \label{eq:sum-par}
					\begin{aligned}			
						\defecterr[\pos,n]{\|} &= 
						\tau^2 \sum_{j=0}^{n-1}  \alpha_j \beta_j (\varphi_2(\tau\mgFieldMatrix[]) - \tfrac12\varphi_1(\tau\mgFieldMatrix[])) \vel[\perp](0) +  \mathcal{O}(\tau^2),
						\\
						\alpha_j &= (\tfrac{\tau}{2}\psi_+(0) + \tn[n-j-1])	\propar\elField'(\tpos[j+1/2]), 
						\quad \beta_j = e^{j\tau\mgFieldMatrix[]} \properp	.	
					\end{aligned}
				\end{equation}
				Here, we apply summation by parts
				\begin{align*}
					\sum_{j=0}^{n-1}  \alpha_j \beta_j  =
					\alpha_{n-1} \sum_{j=0}^{n-1} \beta_j
					+ \sum_{j=0}^{n-2} (\alpha_j - \alpha_{j+1}) \sum_{k=0}^j \beta_k.
				\end{align*}          
				Using the definition of 
				$\funexpGauss_n$ in \eqref{eq:upsilonfilter} we have
				\begin{equation*}
					\sum_{k=0}^j \beta_j = \sum_{k=0}^j e^{k\tau\mgFieldMatrix[]} \properp= \funexpGauss_{j+1}(\tau\mgFieldMatrix[])\properp,
				\end{equation*}
				and conclude
				\begin{align*}
					\alpha_{j+1}-\alpha_j = \tau \propar\elField'(\tpos[j+3/2]) + (\tfrac{\tau}{2}\psi_+(0) + \tn[n-j-1])\propar\bigl(\elField'(\tpos[j+3/2]) - \elField'(\tpos[j+1/2])\bigr),
				\end{align*}
				where we can write
				\begin{align*}
					\elField'(\tpos[j+3/2]) - \elField'(\tpos[j+1/2])=
					\int_{t_j}^{t_{j+1}}\elField''\bigl(\pos(s+\tfrac{\tau}{2})\bigr)[\vel(s+\tfrac{\tau}{2}) ,
					\cdot] ds.
				\end{align*}
				As $E\in C^2$ on a $\normempty{\mgField[]}$ independent domain, we get $\alpha_{j+1}-\alpha_j = \mathcal{O}(\tau)$ and together with Lemma~\ref{lem:rhon-bound}, this implies with $\alpha_{n-1} = \tfrac{\tau}{2}\psi_+(0)
				\propar\elField'(\tpos[n-1/2])$
				\begin{equation}\label{eq:summ-by-parts}
						\tau^2 \sum_{j=0}^{n-1} \alpha_j\beta_j = \tau^2 \alpha_{n-1}\funexpGauss_n(\tau\mgFieldMatrix[]) \properp+ \tau^2\sum_{j=0}^{n-2}(\alpha_{j}-\alpha_{j+1})\funexpGauss_{j+1}(\tau\mgFieldMatrix[]) \properp= \mathcal{O}(\tau^2).
					\end{equation}
					Finally, employing $\normempty{\varphi_2(\tau\mgFieldMatrix[]) - \tfrac12\varphi_1(\tau\mgFieldMatrix[])}\leq 1$ in  \eqref{eq:sum-par} 
					proves the claim for the position component. The velocity component can be treated analogously, as
					\begin{align*}
						\defecterr[\vel,n]{\|} = \propar \tau \sum_{j=0}^{n-1} e^{(n-j-1)\tau\mgFieldMatrix[]}\defectVel[j+1]= \tau \sum_{j=0}^{n-1} \propar \defectVel[j+1]' 
					\end{align*}
					and we can thus repeat the argument above with $\alpha_j = \propar\elField'\bigl(\tpos[j+1/2]\bigr)$.
				\end{proof}

				\begin{lemma}\label{lem:error-all}
					Suppose that \eqref{eq:tauB-lower-bound} holds and that the filter functions in Tables~$\ref{tab:phipmPhi01}$ and $\ref{tab:psichipmPsi01}$ are chosen such that $\Phi_0=\Psi_0=0$ and $\phifilter_\pm$ bounded. Then the defect contribution of the perpendicular part of the position error in \eqref{eq:voc-discrete} is given by
					\begin{align} \label{eq:errqperp}
						\defecterr[\pos,n]{\perp} = \begin{pmatrix}
							\properp & 0
						\end{pmatrix}\defecterr[n]{} = {\tau^2} \properp \psifilter_1 \sum_{j=0}^{n-1} \telField[j+1/2] + \mathcal{O}(\tau^2).
					\end{align}
					Hence, for $\Psi_1 = 0$, we get $\defecterr[\pos,n]{\perp} = \mathcal{O}(\tau^2)$.
				\end{lemma}
				\begin{proof}
					With Lemma~\ref{lem:systemmatrix-power}, $\psifilter_0 = 0$, and \eqref{eq:defectPos} we can write 
					\begin{align*}
						\defecterr[\pos,n]{\perp} &= \properp \tau \sum_{j=0}^{n-1} \bigl(\defectPos[j+1] + \tau \bigl(\tfrac12\psifilter_+ +  \upsilonfilter_{n-j-1}\bigr)\defectVel[j+1]\bigr)\\
						& = \tau^2  \properp\sum_{j=0}^{n-1} \bigl(\psifilter_1 \telField[j+1/2] + \bigl(\tfrac12\psifilter_+ +  \upsilonfilter_{n-j-1}\bigr)\defectVel[j+1]\bigr)+ \mathcal{O}(\tau^2)
					\end{align*}
					From Lemma~\ref{lem:psi-psi+exp},  \eqref{eq:tauB-lower-bound}, and Table~\ref{tab:psichipmPsi01} we have for $x = \pm\tau\normempty{\mgField[]}$
					\begin{equation} \label{eq:pos-perp-decay}
						\abs{\Upsilon_n (ix)} = \abs{\varphi_1(ix) \funexpGauss_n(ix)}
						= \bigl|\frac{e^{inx}-1}{ix}\bigr|
						\leq \frac{2}{\abs{x}} \leq \frac{2}{\vartheta_0}
						\quad \text{and} \quad
						\abs{\psi_+(ix)} \leq \min\{ 1, \tfrac{4}{\abs{x}}\} .
					\end{equation}
					Together with \eqref{eq:defectVel-all}, $\Phi_0=0$, and the bound on the $\phifilter_\pm$, we obtain $\properp\bigl(\tfrac12\psifilter_+ +  \upsilonfilter_{n-j-1}\bigr)\defectVel[j+1] = \mathcal{O}(\tau)$, which gives \eqref{eq:errqperp}.
				\end{proof}
				
				\begin{theorem} \label{thm:second-order}
					Let Assumptions~$\ref{ass:nonres}$ and $\ref{ass:efield-assump}$ hold and assume that the filter functions are chosen such that $\Phi_0 = \Psi_0=\Psi_1=\Phi_1(0)=0$ and that a constant $C_{\auxMatrix}$ exists such that  $\normempty{\auxMatrix} \leq C_{\auxMatrix}$.
					Moreover, let $\lipschitzelField$ be the Lipschitz constant of $\elField$ and assume $\tau$ be sufficiently small s.t. $2\tau (1+T) \lipschitzelField C_{\auxMatrix} < 1$. Then for $0\leq \tn[n]\leq T$ the error 
					of the approximation \eqref{eq:gen-expint} to the solution of \eqref{eq:eul-lag-filt-all} satisfies
					\begin{equation} \label{eq:second-order}
						\normempty{\errorPos[n]} = \mathcal{O}(\tau^2), \qquad \normempty{\propar\errorVel[n]} = \mathcal{O}(\tau^2), \quad \text{and} \quad \normempty{\properp\errorVel[n]} = \mathcal{O}(\tau).
					\end{equation}
				\end{theorem}
				\begin{proof}
					From the error recursion \eqref{eq:voc-discrete} and Lemma~\ref{lem:systemmatrix-power} we have
					\begin{equation*}
						\normempty{\errorPos[n]} \leq 2 \tau \lipschitzelField C_{\auxMatrix} (1+T) \sum_{j=0}^{n}\normempty{\errorPos[j]} + \normempty{\defecterr[\pos,n]{}}.
					\end{equation*}
					Lemmas~\ref{lem:error-par} and \ref{lem:error-all} yield $\normempty{\defecterr[\pos,n]{}} = \mathcal{O}(\tau^2)$. For the parallel velocity error the statements follows similarly from the bound on the position error and Lemma~\ref{lem:error-par}. In Theorem~\ref{thm:error-bound}, we proved $\normempty{\errorVel[n]}= \mathcal{O}(\tau)$, hence this bound also holds for $\normempty{\properp\errorVel[n]}$.
				\end{proof}
				For the filter functions \eqref{eq:HLW-all}, Theorem~\ref{thm:second-order} contains \cite[Thm.~3.1]{HaiLW20} as a special case for a constant magnetic field. The second-order bound proven there
				requires a stepsize restriction coupling $\tau$ to the inverse strength of the
				magnetic field, i.e., $\tau \normempty{\mgField} \lesssim 1$. In contrast, our smallness
				condition $2\tau(1+T)\lipschitzelField C_\auxMatrix < 1$ only depends on the filter functions $\chi_\pm$, $\phi_\pm$ evaluated at $\tau\mgFieldMatrix[]$. For the choices 
				\eqref{eq:voc-v-update-full=midpoint}, \eqref{eq:voc-v-update-full=trap}, and all choices in Table~\ref{tab:psichipmPsi01}, $C_\auxMatrix$ is independent of $\mgField$, cf.~Corollary~\ref{cor:CS-bound}.
				In contrast, for $\phi_\pm$ given in \eqref{eq:voc-v-update-full=HLW}, the
				magnetic field enters via the nonresonance Assumption~\ref{ass:nonres}, which  bounds $C_\auxMatrix$ in terms of $\nonres$ by Lemma~\ref{lem:rhon-bound}. Since the constants in \eqref{eq:second-order} are independent of $\normempty{\mgField}$, Theorem~\ref{thm:second-order} supplies error bounds uniformly in the strength of the magnetic field, in
				particular in the large-stepsize regime $\tau \normempty{\mgField} \gg 1$.
				
				\subsection{Detailed investigation of the velocity error} \label{subsec:leading-terms}
				
				In this section we investigate the perpendicular part of velocity error, which is only of first order by Theorem~\ref{thm:second-order} in order to fully understand the error behavior for large $\tau\normempty{\mgField}$. By \eqref{eq:voc-discrete}, we have
				\begin{equation*}
					\errorVel[n] = \tfrac{\tau}{2} \sum_{j=0}^{n-1} e^{t_{n-j-1} \mgFieldMatrix[]}
					\bigl( \phifilter_- \errorelField[j] + \phifilter_+ \errorelField[j+1]  \bigr)
					+ \defecterr[\vel,n]{}.
				\end{equation*} 
				The first term is in $\mathcal{O}(\tau^2)$ since $\elField$ is Lipschitz and ${\errorPos[n]} = \mathcal{O}(\tau^2)$ and so $	\defecterr[\vel,n]{\|}$ . Hence, by \eqref{eq:defect-v-direct}, we need to study the perpendicular part of
				\begin{equation}  \label{eq:Wn}
					\begin{aligned}
						\defecterr[\vel,n]{} 
						&= W_n +  \tau^2 \sum_{j=0}^{n-1} e^{t_{n-j-1} \mgFieldMatrix[]} \defectVel[j+1]' + \mathcal{O}(\tau^2),
						\\ 
						W_n &= \tau^2 \Phi_1(\tau\mgFieldMatrix[])\sum_{j=0}^{n-1} e^{\tn[n-j-1]\mgFieldMatrix[]} \elField'(\tpos[j+1/2]) \rem_{j+1/2}.	
					\end{aligned}
				\end{equation}
				First we show that $W_n= \mathcal{O}(\tau^2)$, so that the dominant term is second one.
				\begin{lemma}[Error contribution of $W_n$] \label{lem:boundWn}
					Suppose that the filter functions are chosen such that $\Phi_0 = 0$ and $\Phi_1(0) = 0$ hold in \eqref{eq:defectVel-first} and \eqref{eq:phi1}, respectively and that the nonresonance Assumption $\ref{ass:nonres}$ holds. Then we have
					\begin{equation*}
						\propar W_n = 0 \quad \text{and} \quad \properp W_n = \mathcal{O}(\tau^2).
					\end{equation*}
				\end{lemma}
				To be more precise, the first equation only requires $\Phi_1(0)=0$ but not Assumption~\ref{ass:nonres}, while for the second, it is the other way around.
				\begin{proof}
					The first identity follows immediately from $\Phi_1(0) = 0$ and \eqref{eq:Wn}.
					
					For the orthogonal projection, we write
					\begin{equation*}
						\properp W_n = \tau^2 \properp \Phi_1(\tau\mgFieldMatrix[]) \sum_{j=0}^{n-1} \beta_j \alpha_j,
						\qquad
						\alpha_j = \elField'(\tpos[j+1/2])\rem_{j+1/2}, \quad \beta_j = e^{\tn[n-j-1]\mgFieldMatrix[]}.
					\end{equation*}
					Summation by parts implies 
					\begin{align}  \label{eq:sum-by-parts}
						\sum_{j=0}^{n-1} \beta_j \alpha_j  &=
						\sum_{j=0}^{n-1} \beta_j \alpha_{n-1} 
						+ \sum_{j=0}^{n-2} \sum_{k=0}^j \beta_k (\alpha_j - \alpha_{j+1}) \\
						& = \funexpGauss_{n}(\tau \mgFieldMatrix[]) \alpha_{n-1} 
						+ e^{\tn[n-1]\mgFieldMatrix[]}\sum_{j=0}^{n-2} \funexpGauss_{j}(-\tau \mgFieldMatrix[]) (\alpha_j - \alpha_{j+1}),
						\nonumber
					\end{align}          
					where we used \eqref{eq:upsilonfilter} in the last identity.
					Next, we apply the variation-of-constants formula \eqref{eq:exact-q-v-voc-b} which gives
					\begin{align*}
						\rem_{j+1/2} &= \tvel[j+1/2] - e^{\tn[j+1/2]\mgFieldMatrix[]}\vel[\perp](0) \\
						&= e^{\tn[j+1/2]\mgFieldMatrix[]}\vel[\|](0) + \int_0^{\tn[j+1/2]} e^{(\tn[j+1/2]-s)\mgFieldMatrix[]}\telField(s)ds \\& = \vel[\|](0) + \int_0^{\tn[j+1/2]} e^{(\tn[j+1/2]-s)\mgFieldMatrix[]}\telField(s)ds.
					\end{align*}
					From this identity we conclude 
					\begin{align*}
						\alpha_{j} - \alpha_{j+1} &= \elField'(\tpos[j+1/2])\Bigl(\int_0^{\tn[j+1/2]} e^{(\tn[j+1/2]-s)\mgFieldMatrix[]}\telField(s)ds - \int_0^{\tn[j+3/2]} e^{(\tn[j+3/2]-s)\mgFieldMatrix[]}\telField(s)ds\Bigr) \\ 
						& \quad +  \bigl(\elField'(\tpos[j+1/2]) - \elField'(\tpos[j+3/2])\bigr)\rem_{j+3/2} \\
						& = \elField'(\tpos[j+1/2])
						(\Id - e^{\tau\mgFieldMatrix[]})\int_0^{\tn[j+1/2]} e^{(\tn[j+1/2]-s)\mgFieldMatrix[]}\telField(s)ds
						+ \mathcal{O}(\tau),
					\end{align*}
					since $E\in C^2$ on a $\normempty{\mgField[]}$ independent domain and $\rem_{j+3/2}$ is bounded.
					Hence, we get from yet another integration by parts
					\begin{align*}
						\alpha_{j} - \alpha_{j+1} 
						&=  \tau\elField'(\tpos[j+1/2])\varphi_{1}(\tau\mgFieldMatrix[])e^{\tn[j+1/2]\mgFieldMatrix[]}\int_0^{\tn[j+1/2]} (-\mgFieldMatrix[])e^{-s\mgFieldMatrix[]}\telField(s)ds + \mathcal{O}(\tau)\\
						&=
						\tau\elField'(\tpos[j+1/2])\varphi_{1}(\tau\mgFieldMatrix[])e^{\tn[j+1/2]\mgFieldMatrix[]}
						\Bigl(  e^{-\tn[j+1/2]\mgFieldMatrix[]} \telFieldJac{1}(\tn[j+1/2]) \\&\quad-
						\telFieldJac{1}(0) - \int_0^{\tn[j+1/2]} e^{-s\mgFieldMatrix[]}\telFieldJac{1}(s)ds 
						\Bigr)+ \mathcal{O}(\tau)\\
						&= \mathcal{O}(\tau).
					\end{align*}
					The claim now follows from nonresonance condition which ensures \eqref{eq:funexpgauss-bound}.	
				\end{proof}
				
				In order to determine the leading error term we make use of the diagonalization \eqref{eq:mgmatrix-diag} of $\mgFieldMatrix[]$.
				
				\begin{lemma}[Perpendicular velocity error] \label{lem:perp-vel-err}
					Suppose that the filter functions defined in 
					\eqref{eq:defectVel-first} and 	\eqref{eq:phi1} are such that $\Phi_0 = 0$ and $\Phi_1(0)=0$. Furthermore, assume that a nonresonance condition $\bigl|\tau \normempty{\mgField[]} - \pi k \bigr| \geq \nonres_2> 0$ holds for all $k\in \mathbb{Z}\setminus\{0\}$. Then the defect contribution of the perpendicular part of the velocity error in \eqref{eq:voc-discrete} behaves like 
					\begin{subequations}\label{eq:Phiperp-all}
						\begin{align}
							\defecterr[\vel,n]{\perp}  &= \begin{pmatrix}
								0 &\properp 
							\end{pmatrix}\defecterr[n]{}  = \tau\Phi_\perp(\tau\mgFieldMatrix[])e^{\tn[n] \mgFieldMatrix[]} D_{E}\vel[\perp](0) + \mathcal{O}(\tau^2),
							\label{eq:Phiperp-a}
							\intertext{with}
							D_{E} &= \tau\sum_{j=0}^{n-1} U_\perp \mathrm{diag}\bigl(U_\perp^* \elField'(\tpos[j+1/2]) U_\perp\bigr) U_\perp^*, \label{eq:Phiperp-b}\\ \Phi_\perp(z) &= \frac{1}{z}\bigl(1-\sinch(\tfrac{z}{2})\bigr) - \tfrac{1}{2} \varphi_1(-z)\bigl(\phi_+(z) - \varphi_{1}(\tfrac{z}{2})\bigr),\label{eq:Phiperp-c}
						\end{align}
					\end{subequations}
					where $\mathrm{diag}$ denotes the diagonal of a matrix and $U_\perp$ was defined in \eqref{eq:mgmatrix-diag}.
				\end{lemma}
				\begin{proof}
					From \eqref{eq:Wn}, Lemma \ref{lem:boundWn}, and \eqref{eq:defect-v-strich}  we have 
					\begin{align}
						\defecterr[\vel,n]{\perp} &= \tau \properp  \sum_{j=0}^{n-1} e^{(n-j-1)\tau \mgFieldMatrix[]}\defectVel[j+1] + \mathcal{O}(\tau^2)
						\label{eq:defecterr-vel-perp}
						\\
						&= \tau^2 \properp  \sum_{j=0}^{n-1} e^{\tn[n-j-1] \mgFieldMatrix[]}\Bigl(\int_0^1  (\sigma-\tfrac12) e^{\tau(1-\sigma)\mgFieldMatrix[]}
						\elField'(\tpos[j+1/2])
						\varphi_1\bigl( (\sigma-\tfrac12) \tau \mgFieldMatrix[]\bigr)d\sigma
						\nonumber\\
						&\quad - \tfrac{1}{4} \bigl(\phifilter_+  \elField'(\tpos[j+1/2])
						\varphi_1(\tfrac{\tau}{2}\mgFieldMatrix[])
						- \phifilter_-  \elField'(\tpos[j+1/2])
						\varphi_1(-\tfrac{\tau}{2}\mgFieldMatrix[])\bigr)\Bigr)
						e^{\tn[j+1/2]\mgFieldMatrix[]} \vel[\perp](0)\nonumber  \\&\quad+ \mathcal{O}(\tau^2). 
						\nonumber
					\end{align}
					Having the projection $\properp$ on the left and right of the leading order term, implies that the expression in between is de facto two-dimensional. We write $F_j :=  U_\perp^* \elField'(\tpos[j+1/2]) U_\perp$, and the above expression becomes 
					\begin{align*}
						\defecterr[\vel,n]{\perp} = U_\perp G_n U_\perp^*\vel[\perp](0)
					\end{align*}
					with $G_n\in \mathbb{C}^{2\times 2}$ defined by
					\begin{align*}
						&G_n = \tau \sum_{j=0}^{n-1} e^{\tn[n-j-1]\mgFieldMatrix[\perp]}\tau\Bigl(\int_0^1  (\sigma-\tfrac12) e^{(1-\sigma)\tau\mgFieldMatrix[\perp]}
						F_j
						\varphi_1\bigl( (\sigma-\tfrac12) \tau\mgFieldMatrix[\perp]\bigr)d\sigma
						\nonumber\\
						&\quad - \tfrac{1}{4} \bigl(\phi_+(\tau\mgFieldMatrix[\perp]) F_j
						\varphi_1(\tfrac{\tau}{2}\mgFieldMatrix[\perp])
						- \phi_-(\tau\mgFieldMatrix[\perp])  F_j
						\varphi_1(-\tfrac{\tau}{2}\mgFieldMatrix[\perp])\bigr)\Bigr)
						e^{\tn[j+1/2]\mgFieldMatrix[\perp]}. 
					\end{align*}
					Note that all $\varphi_k$ are analytic, which implies that
					\begin{align*}
						\varphi_k(\tau\mgFieldMatrix[\perp]) = \begin{pmatrix}
							\varphi_k(\omega)&0\\0&\varphi_k(-\omega)
						\end{pmatrix},\qquad \omega = i\tau\normempty{\mgField[]}.
					\end{align*}
					Thus the $(1,1)$-entry $G_n^{(1,1)}$ is given by
					\begin{align*}
						G_n^{(1,1)} &= \tau^2 \sum_{j=0}^{n-1}e^{(n-1/2)\omega} \Bigl(\int_0^1 (\sigma-\tfrac12) e^{(1-\sigma)\omega}\varphi_1\bigl( (\sigma-\tfrac12) \omega\bigr) d\sigma \\&\quad- \tfrac{1}{4} \bigl(\phi_+(\omega)
						\varphi_1(\tfrac{\omega}{2})
						- \phi_-(\omega)
						\varphi_1(-\tfrac{\omega}{2})\bigr)\Bigr) F_j^{(1,1)}\\
						& = \tau^2 e^{n\omega}\sum_{j=0}^{n-1} \Bigl(\frac{1}{\omega}\bigl(1-\sinch(\tfrac{\omega}{2})\bigr) \\&\quad- \tfrac{1}{4} e^{-\tfrac{\omega}{2}}\bigl(\phi_+(\omega)
						\varphi_1(\tfrac{\omega}{2})
						- \phi_-(\omega)
						\varphi_1(-\tfrac{\omega}{2})\bigr)\Bigr) F_j^{(1,1)}\\
						& = \tau^2 \Phi_\perp(\omega) e^{n\omega}\sum_{j=0}^{n-1}  F_j^{(1,1)},
					\end{align*}
					where we used \eqref{eq:id-sinch}, \eqref{eq:id-sinch-exp}, and \eqref{eq:id-phi1-exp} together with  $\Phi_0=0$, which implies $\phi_- = 2\varphi_1 - \phi_+$.  Similarly we find
					\begin{align*}
						G_n^{(1,2)} &= \tau^2 \sum_{j=0}^{n-1} e^{(n-3/2)\omega}\Bigl(\int_0^1 (\sigma-\tfrac12) e^{(1-\sigma)\omega}\varphi_1\bigl( (\tfrac12-\sigma) \omega\bigr) d\sigma \\&\quad- \tfrac{1}{4} \bigl(\phi_+(\omega)
						\varphi_1(-\tfrac{\omega}{2})
						- \phi_-(\omega)
						\varphi_1(\tfrac{\omega}{2})\bigr)\Bigr)e^{-2j\omega} F_j^{(1,2)}\\
						& = \tau^2 \sum_{j=0}^{n-1} e^{(n-2j-1)\omega}\Bigl(\tfrac{1}{\omega}\bigl(\sinch(\tfrac{\omega}{2})- \sinch(\omega)\bigr) \\&\quad- \tfrac{1}{4} e^{-\tfrac{\omega}{2}} \bigl(\phi_+(\omega)
						\varphi_1(-\tfrac{\omega}{2})
						- (2 \varphi_1(\omega) - \phi_+(\omega))
						\varphi_1(\tfrac{\omega}{2})\bigr)\Bigr) F_j^{(1,2)}\\
						& = \tau^2 \sum_{j=0}^{n-1} e^{(n-2j-1)\omega}\varphi_1(-\omega)\Bigl(\tfrac{1}{\omega}\bigl(e^{\tfrac{\omega}{2}} - \tfrac12(1 + e^\omega)\bigr) \\&\quad- \tfrac{1}{2} \bigl(\phi_+(\omega)
						- e^{\tfrac{\omega}{2}}\varphi_1(\tfrac{\omega}{2})\bigr)\Bigr) F_j^{(1,2)}\\
						& = \tau^2 \sum_{j=0}^{n-1} e^{(n-2j-1)\omega}  \tfrac12 \varphi_1(-\omega)		
						\bigl(\varphi_1(\omega) - \phi_+(\omega) \bigr) F_j^{(1,2)}
					\end{align*}
					and the remaining two entries can be obtained similarly replacing $\omega$ with $-\omega$. In particular, we have
					\begin{align*}
						G_n^{(2,2)} = \tau^2 \Phi_\perp(-\omega) e^{-n\omega}\sum_{j=0}^{n-1}  F_j^{(2,2)}
					\end{align*}
					and therefore
					\begin{align*}
						\mathrm{diag}(G_n) = \tau^2 \Phi_\perp(\tau\mgFieldMatrix[\perp]) e^{\tn[n]\mgFieldMatrix[\perp]}\sum_{j=0}^{n-1}  \mathrm{diag}(F_j).
					\end{align*}
					For the off-diagonal entry, we have
					\begin{align} \label{eq:phi_perp-tilde}
						G_n^{(1,2)} 
						&= \tau^2 e^{(n-1)\omega} \widetilde{\Phi}_\perp(\omega) \sum_{j=0}^{n-1} e^{-2j\omega} F_j^{(1,2)}\nonumber
						\intertext{with} \widetilde{\Phi}_\perp(\omega) &=  \tfrac12 \varphi_1(-\omega)		
						\bigl(\varphi_1(\omega) - \phi_+(\omega) \bigr)
					\end{align}
					and summation by parts \eqref{eq:sum-by-parts} with $\alpha_j = F_j^{(1,2)}$ and $\beta_j = e^{-2j\omega}$
					gives 
					\begin{align*}
						G_n^{(1,2)} &= \tau^2 e^{(n-1)\omega} \widetilde{\Phi}_\perp(\omega)\sum_{j=0}^{n-1}\beta_j\alpha_j \\&= \tau^2 e^{(n-1)\omega} \widetilde{\Phi}_\perp(\omega)\Bigl(\funexpGauss_{n}(-2 \omega) \alpha_{n-1} 
						+ \sum_{j=0}^{n-2} \funexpGauss_{j}(-2 \omega) (\alpha_j - \alpha_{j+1})\Bigr).
					\end{align*}
					As in \eqref{eq:summ-by-parts}, $a_{j}-a_{j+1} = \mathcal{O}(\tau)$, so if $|\omega|$ is bounded away from integer multiples of $\pi$, we obtain $G_n^{(1,2)} = \mathcal{O}(\tau^2)$.
				\end{proof}

				We now derive the explicit expressions for $\Phi_\perp$.
				\begin{lemma} \label{lem:vel-perp-lead-order}
					For \eqref{eq:voc-v-update-full=HLW}, \eqref{eq:voc-v-update-full=midpoint}, and \eqref{eq:voc-v-update-full=trap}, the leading order term in \eqref{eq:Phiperp-all} is given in Table~$\ref{tab:phipmPhi01}$.
				\end{lemma}
				\begin{proof}
					For \eqref{eq:voc-v-update-full=HLW}, we have $\phi_+(z) = 2\tfrac{\varphi_2(-z)}{\varphi_1(-z)}$. Hence, \eqref{eq:Phiperp-c} becomes 
					\begin{align*}
						\Phi_\perp(z) &= \frac{1}{z}\bigl(1-\sinch(\tfrac{z}{2})\bigr) - \tfrac{1}{2} \varphi_1(-z)\bigl(2\tfrac{\varphi_2(-z)}{\varphi_1(-z)} - \varphi_{1}(\tfrac{z}{2})\bigr)\\
						&=\frac{1}{z}\bigl(1-\sinch(\tfrac{z}{2})\bigr) -  \bigl(\varphi_2(-z) - \tfrac{1}{z} \sinch(\tfrac{z}{2})(1-e^{-\tfrac{z}{2}})\bigr)\\
						&=\frac{1}{z}\bigl(1- \sinch(\tfrac{z}{2}) - ((1-\varphi_1(-z)) - \sinch(\tfrac{z}{2}) + \sinch(\tfrac{z}{2}) e^{-\tfrac{z}{2}})\bigr)
					\end{align*}
					using \eqref{eq:id-phi1phi1half} and \eqref{eq:id-sinch-exp}. The claim follows from \eqref{eq:id-sinch-exp}.
					
					This implies that $\Phi_\perp$ for other choices of filter functions can be written as
					\begin{align*}
						\Phi_\perp(z) = -\tfrac12 \varphi_1(-z)\bigl(\phi_+(z) - 2\tfrac{\varphi_2(-z)}{\varphi_1(-z)}\bigr) = \varphi_2(-z) - \tfrac12\varphi_1(-z) \phi_+(z).
					\end{align*}
					For \eqref{eq:voc-v-update-full=midpoint}, we have $\phi_+(z) = \varphi_{1}(z)$, so 
					\begin{align*}
						\Phi_\perp(z) = \varphi_2(-z) - \tfrac12\varphi_1(-z) \varphi_1(z) = - \psi_{1,{\sinh}}(z)
					\end{align*}
					using \eqref{eq:id-phi1-phi2} and \eqref{eq:id-psi1sinh}.
					
					For \eqref{eq:voc-v-update-full=trap}, we have $\phi_+(z) = 2\bigl(\varphi_1(z)-\varphi_2(z)\bigr)$ so
					\begin{align*}
						\Phi_\perp(z) &= \varphi_2(-z) - \varphi_1(-z) \bigl(\varphi_1(z)-\varphi_2(z)\bigr) \\&= - \psi_{1,{\sinh}}(z) - \sinch(\tfrac{z}{2})e^{-\tfrac{z}{2}} \bigl(\tfrac12\varphi_1(z)-\varphi_2(z)\bigr)
					\end{align*}
					using the derivation for \eqref{eq:voc-v-update-full=midpoint} and \eqref{eq:id-sinch-exp}. The claim then follows from \eqref{eq:id-phi1o2mphi2}.
				\end{proof}
				The derivation above implies that the filter functions in \eqref{eq:HLW-all} from \cite[eq.~(2.12)]{HaiLW20} are the unique choice that satisfy $\Phi_\perp = 0$ and thus $\defecterr[\vel, n]{} = \mathcal{O}(\tau^2)$. In view of Theorem~\ref{thm:second-order}, these are the unique filter functions that yield $\normempty{\error[n]} = \mathcal{O}(\tau^2)$ if the nonresonance Assumption~\ref{ass:nonres} is satisfied.

				Our analysis as well as the experiments show that resonance effects cannot be eliminated in general. However, in the special case, that $\elField'(\tpos[n+1/2])$ commutes with $\mgFieldMatrix[]$ for all $n$, they vanish.
				
				\begin{theorem} \label{thm:commute-no-res}
					Assume that $\elField'(\tpos[n+1/2])$ commutes with $\mgFieldMatrix[]$ for all $n$. Then Theorem~$\ref{thm:second-order}$ and Lemma~\ref{lem:perp-vel-err} hold without the nonresonance condition in Assumption~$\ref{ass:nonres}$.	
				\end{theorem}
				\begin{proof}
					With the new assumption that $\elField'(\tpos[n+1/2])$ commutes with $\mgFieldMatrix[]$, we can give the following more favorable representation for the velocity defect than in  Lemma~\ref{lem:defect-v-direct}, namely
					\begin{align}
						\defectVel[n+1]' & = \Bigl(\int_0^1  (\sigma-\tfrac12) e^{\tau(1-\sigma)\mgFieldMatrix[]}
						\varphi_1\bigl( (\sigma-\tfrac12) \tau \mgFieldMatrix[]\bigr)d\sigma
						\nonumber\\
						&\quad
						- \tfrac{1}{4} \bigl(\phifilter_+ 
						\varphi_1(\tfrac{\tau}{2}\mgFieldMatrix[])
						- \phifilter_-  
						\varphi_1(-\tfrac{\tau}{2}\mgFieldMatrix[])\bigr)
						\Bigr)
						e^{\tn[n+1/2]\mgFieldMatrix[]} \properp 
						\elField'(\tpos[n+1/2])  \vel[\perp](0)   \nonumber\\
						&=\Phi_\perp(\tau\mgFieldMatrix[]) e^{\tn[n+1]\mgFieldMatrix[]} \properp 
						\elField'(\tpos[n+1/2])  \vel[\perp](0) 	
						\label{eq:defect-v-strich-commute}
					\end{align}
					with $\Phi_\perp$ defined in \eqref{eq:Phiperp-c}.
					Hence, if $\phi_+(0) = 1$, which is the case for all choices in Table~\ref{tab:phipmPhi01}, then $\Phi_\perp(0)=0$. Therefore, for the parallel error contributions in Lemma~\ref{lem:error-par}, we obtain $\defecterr[\pos,n]{\|} = \mathcal{O}(\tau^2)$ and $\defecterr[\vel,n]{\|} = \mathcal{O}(\tau^2)$ immediately without a nonresonance condition. Moreover, for the perpendicular velocity error, instead of the technique used in the proof of Lemma~\ref{lem:perp-vel-err}, we can directly insert \eqref{eq:defect-v-strich-commute} into \eqref{eq:defecterr-vel-perp} to obtain
					\begin{align*}
						\defecterr[\vel,n]{\perp} &= \tau^2 \Phi_\perp(\tau\mgFieldMatrix[])\sum_{j=0}^{n-1} e^{\tn[n-j-1] \mgFieldMatrix[]} e^{\tn[j+1]\mgFieldMatrix[]}\properp \elField'(\tpos[j+1/2])\vel[\perp](0)  + \mathcal{O}(\tau^2)\\
						&=\tau^2 \Phi_\perp(\tau\mgFieldMatrix[])  e^{\tn[n] \mgFieldMatrix[]}\sum_{j=0}^{n-1}\properp \elField'(\tpos[j+1/2])\vel[\perp](0)  + \mathcal{O}(\tau^2).
					\end{align*}
					As for the parallel defects, we immediately obtain the result of Lemma~\ref{lem:perp-vel-err} without a nonresonace condition.	
				\end{proof}
				
				\section{Numerical examples} \label{sec:numerics}
				We illustrate the behavior of the exponential integrator \eqref{eq:gen-expint} for different choices of filter functions and visualize the error bounds of the previous sections.\footnote{At \url{https://gitlab.kit.edu/kit/ianm/ag-numerik/projects/filtered-boris} we provide access to the code used to conduct the experiments in this section. Parts of the Code and plotting scripts were written with the help of AI, i.e., Claude Code using Claude Opus 4.8 and 5.} In particular, we investigate robustness at resonances, the convergence order across different $\tau\normempty{\mgField[]}$ regimes, the distinct behavior of the error components parallel and perpendicular to $\mgField[]$, and the explicit leading-order terms derived in Lemmas~\ref{lem:error-all}, \ref{lem:perp-vel-err}, and \ref{lem:vel-perp-lead-order}.
				
				For all experiments, we use \eqref{eq:eul-lag-filt-all} with
				\begin{align*}
					\pos[0] = \begin{pmatrix}
						1\\2\\3
					\end{pmatrix},\quad \vel[0]  = \begin{pmatrix}
						1\\0\\0
					\end{pmatrix},
					\quad
					\text{ and } 
					\quad 
					\mgField[] = \frac{1}{\delta}\frac{2}{\sqrt{21}}
					\begin{pmatrix}
						1\\2\\ \tfrac12
					\end{pmatrix}, 
					\quad \delta > 0.
				\end{align*}
				Hence we have $\normempty{\mgField[]} = 1/\delta$. We choose a smooth, anisotropic, electric potential
				\begin{equation*}
					V(\pos) = \frac{1}{\alpha^2} \sin^2\Bigl(\alpha \sqrt{\pos^{\top}A \pos}\Bigr),
					\qquad
					A = \mathrm{diag}(1,\tfrac32,\tfrac{7}{10}),
					\qquad \alpha = 10,
				\end{equation*}
				which leads to the electric field 
				\begin{equation} \label{eq:E-nonlin}
					\elField(\pos) = -\nabla V(\pos)
					= -
					2 \sinc(2\alpha \sqrt{\pos^{\top} A \pos}) A\pos.
				\end{equation}
				To compare the qualitative behavior of linear and nonlinear fields we use
				\begin{equation} \label{eq:E-lin}
					\elField_{\mathrm{lin}}(\pos) = -2 A\pos ,
				\end{equation}
				the harmonic field obtained from linearizing $\elField[]$ at $\pos = 0$.
				
				The potential is designed to exercise the full analysis. The genuine nonlinearity with moderate oscillation is chosen to amplify the higher-order derivatives of $\elField$ making $\elField[]$-dependent terms in the error expansion numerically visible. The anisotropy of $A$ prevents the derivative $\elField[]'$ from commuting with the rotation generated by $\mgField[]$, which is necessary for a generic example as shown in Theorem~\ref{thm:commute-no-res}. We visualize this in Figure~\ref{fig:exp5} below.
				
				In all experiments we use a reference solution from an adaptive eighth-order explicit Runge--Kutta method (Dormand--Prince
				\textsf{DP8} scheme) with relative tolerance $10^{-13}$ and dense output; all errors reported below are the maximal error on the time-interval $[0,1]$ in the Euclidean norm.
				\begin{figure}
					\includegraphics{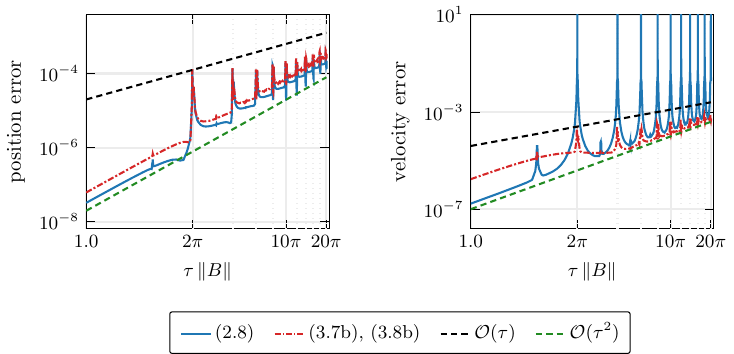}
					\caption{Maximum position (left) and velocity (right) errors
						versus $\theta=\tau\normempty{\mgField[]}=\tau/\delta$ for $\delta=10^{-3}$. Additional sample points are placed near $\pi, 3\pi, 5\pi$ and near the resonances $2k\pi,$ $k=1,\dots,10$, down to a distance of $10^{-5}$ from the respective multiple of $\pi$, to resolve the spikes.}
					\label{fig:exp1}
				\end{figure}
				
				First, we compare the error behavior of the filter functions in \eqref{eq:HLW-all}, corresponding to \cite[eq.~(2.12)]{HaiLW20}, with the choice \eqref{eq:voc-v-update-full=trap}, \eqref{eq:voc-vhalf-update-vdk-neu}. The primary interest here is the behavior in different regimes of $\tau\normempty{\mgField[]}$ especially close to resonances $\tau\normempty{\mgField[]} = 2k\pi$ for $k\in\mathbb{N}$. Recall that the approximations from \eqref{eq:HLW-all} are not defined at resonant time steps because of the singularity of the filter functions $\phi_\pm$. For the position error, the left panel of Figure~\ref{fig:exp1} confirms the convergence behavior established in the preceding sections. Away from the resonances, both filter choices exhibit clean second-order convergence in $\tau$, in agreement with Theorem~\ref{thm:second-order}: the error curves run parallel to the $\mathcal{O}(\tau^2)$ reference line over the entire non-resonant range. At the resonances themselves, the analysis in Theorem~\ref{thm:error-bound} only guarantees first order for \eqref{eq:voc-v-update-full=trap}, \eqref{eq:voc-vhalf-update-vdk-neu}.
				For \eqref{eq:HLW-all}, the singularities caused by the singular filter functions $\phi_\pm$ in $\auxMatrix$ are not visible as for the filter functions \eqref{eq:HLW-all} satisfy
				\begin{align*}
					\tau(\tfrac12 \psifilter_+ +\upsilonfilter_n)\phifilter_\pm 
					= \tau \funexpGauss_{n}(\tau \mgFieldMatrix[])\varphi_1(\tau\mgFieldMatrix[])\phifilter_\pm
					= \tau \funexpGauss_{n}(\tau \mgFieldMatrix[])e^{\pm \tau \mgFieldMatrix[]}\varphi_2(\mp\tau\mgFieldMatrix[]),
				\end{align*}
				using \eqref{eq:id-phi1-exp}, which is bounded by $\tfrac12 t_n$ by \eqref{eq:upsilonfilter}.
				
				While for \eqref{eq:HLW-all} we see the effect of the singularities of the filter functions $\phi_\pm$, and indeed the error deteriorates in narrow spikes around resonances. The observed peaks, however, remain visibly below the $\mathcal{O}(\tau)$ envelope, i.e., the predicted loss of order does not occur; this is a consequence of 
				the precise form of the parallel defect derived in \eqref{eq:peak-order-func} for fixed $\delta = 1/\normempty{\mgField[]}$ we have at resonances $\tau/\delta = 2\pi k$,
				\begin{align*}
					\tau \normempty{\varphi_2(\tau\mgFieldMatrix[]) - \tfrac12\varphi_1(\tau\mgFieldMatrix[]) }
					= \delta.
				\end{align*}
				Finally, for large stepsizes the error is mainly driven by the non-resonant contribution, which grows like $\tau^2$ and eventually dominates the resonance spikes.
				
				The velocity error in the right panel of Figure~\ref{fig:exp1} shows a different picture. Away from the resonances, the filter functions \eqref{eq:HLW-all} again converge with second order, whereas by Theorem~\ref{thm:error-bound}, the choice \eqref{eq:voc-v-update-full=trap}, \eqref{eq:voc-vhalf-update-vdk-neu} leads to errors below the line representing first order for all $\tau$. The more precise error behavior results from the leading order term $\Phi_\perp(\tau \mgFieldMatrix[])$ in Lemma~\ref{lem:perp-vel-err}, which is given explicitly in Table~\ref{tab:phipmPhi01} and visualized in Figure~\ref{fig:exp3} below.

				At the resonances, the two methods differ fundamentally. For \eqref{eq:HLW-all} we see the effect of the singularities of the filter functions $\phi_\pm$,  with error peaks growing like the reciprocal distance of $\tau\normempty{\mgField[]}$ to $2k\pi$.
				The choice \eqref{eq:voc-v-update-full=trap}, \eqref{eq:voc-vhalf-update-vdk-neu}, in contrast, is robust at resonances: no blow-up occurs, and, as for the position component, no first-order behavior is visible. In addition, both methods exhibit smaller humps at odd multiples of $\pi$, most pronounced for \eqref{eq:HLW-all}. As in the left panel, the errors are eventually dominated by the $\mathcal{O}(\tau^2)$ contributions for large stepsizes.
				\begin{figure}
					\includegraphics{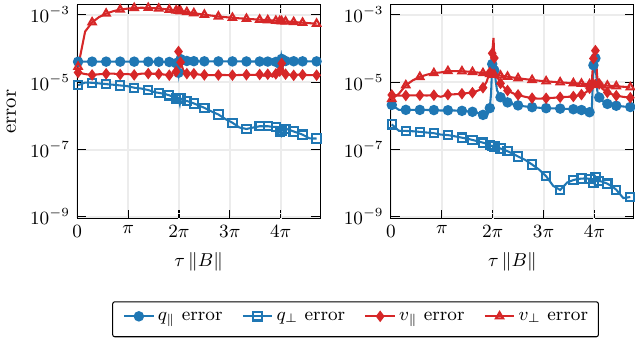}
					\caption{Errors of \eqref{eq:voc-v-update-full=trap}, \eqref{eq:voc-vhalf-update-vdk-neu}, decomposed into the components parallel and perpendicular to $\mgField$, versus $\theta = \tau\normempty{\mgField[]}$ for the linear field \eqref{eq:E-lin} (left) and the nonlinear field \eqref{eq:E-nonlin} (right). The stepsize $\tau = 0.005$ is fixed and $\normempty{\mgField[]}$ is varied.}
					\label{fig:exp2}
				\end{figure}
				
				In Figure~\ref{fig:exp2} we show the contributions of the error of \eqref{eq:voc-v-update-full=trap}, \eqref{eq:voc-vhalf-update-vdk-neu} into components parallel and perpendicular to $\mgField[]$, for the linear field \eqref{eq:E-lin} (left) and the nonlinear field \eqref{eq:E-nonlin} (right). Here the stepsize $\tau = 0.005$ is fixed and $\theta = \tau\normempty{\mgField[]}$ is varied through $\normempty{\mgField[]}$, so that the panels directly display the $\theta$-dependence of the error for fixed $\tau$. The qualitative behavior matches the leading-order terms of Lemmas~\ref{lem:error-par}, \ref{lem:error-all} and \ref{lem:perp-vel-err} componentwise: the parallel position error is essentially independent of $\theta$, the perpendicular position error decays with growing $\normempty{\mgField[]}$, and the resonances at $\tau\normempty{\mgField[]} = 2k\pi$ are visible predominantly in the velocity components. The decay of the perpendicular position error is a consequence of \eqref{eq:pos-perp-decay}, which yields that it decays like $1/(\tau\normempty{\mgField[]})$ for large magnetic fields and shows no sign of resonances. Moreover, we observe that the linear electric field already reproduces all qualitative features of the general case, and the two panels differ only in the size of the error constants.
				
				\begin{figure}
					\includegraphics{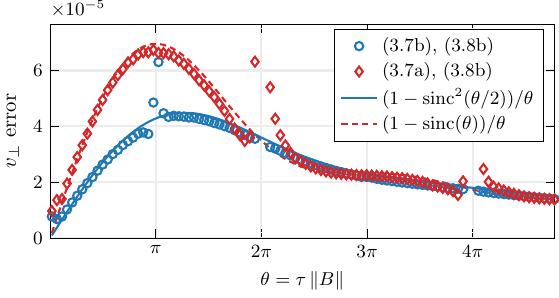}
					\caption{Perpendicular velocity error for $\phi_\pm$ filters \eqref{eq:voc-v-update-full=trap} and \eqref{eq:voc-v-update-full=midpoint} (both with \eqref{eq:voc-vhalf-update-vdk-neu}) versus $\theta = \tau\normempty{\mgField[]}$, together with the leading-order error terms from Lemma~\ref{lem:vel-perp-lead-order}, scaled by a constant to match at $\theta = 1$. For the errors, the stepsize $\tau = 0.01$ is fixed and $\normempty{\mgField[]}$ is varied.}
					\label{fig:exp3}
				\end{figure}
				
				Figure~\ref{fig:exp3} examines the perpendicular velocity error in more detail and compares it to the leading-order error terms derived in Lemma~\ref{lem:vel-perp-lead-order} and Table~\ref{tab:phipmPhi01}.  The measured errors follow these curves closely over the whole range of $\theta$, including the non-monotone behavior up to the first resonance: position, height and shape of the maximum near $\theta \approx \pi$ are correctly predicted, as is the slow decay beyond it. In particular, this confirms that for moderate values of $\theta$ the perpendicular velocity error is completely described by the oscillatory leading-order term, which explains the absence of a clear convergence order in this regime observed in Figure~\ref{fig:exp1}. Deviations from the leading-order curves are visible only close to the resonances $\theta = 2k\pi$, where the resonant terms resulting from the summation-by-parts step in the proof of Lemma~\ref{lem:boundWn} become relevant. The additional resonances at odd multiples of $\pi$ arising from the off-diagonal terms in Lemma~\ref{lem:perp-vel-err} vanish for the $\phi_\pm$ filter functions \eqref{eq:voc-v-update-full=midpoint}, as $\widetilde{\Phi}_\perp = 0$ if $\phi_+ = \varphi_{1}$.	
				\begin{figure}
					\includegraphics{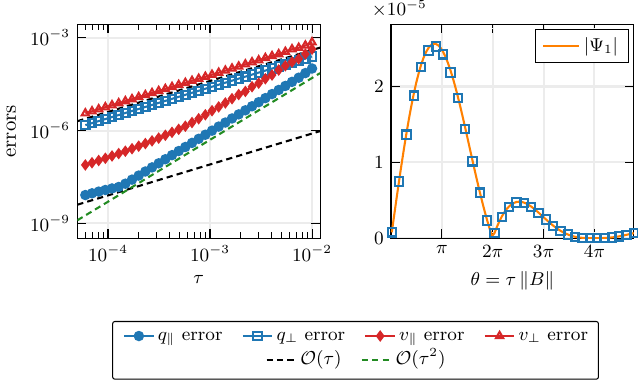}
					\caption{Filter choice \eqref{eq:voc-v-update-full=trap}, \eqref{eq:voc-vhalf-update-vdk} with $\Psi_1 \neq 0$. The electrical field was chosen $20$ times larger for this experiment. Left: error components versus $\tau$ at the fixed non-resonant value $\tau\normempty{\mgField[]} = 3$. Right: perpendicular position error versus $\theta = \tau\normempty{\mgField[]}$ at fixed stepsize $\tau = 0.001$, together with the leading-order term $\Psi_1$ from Lemma~\ref{lem:error-all} given in \eqref{eq:Psi_1_Fig_4} and scaled by a constant to match at $\theta = 1$.}
					\label{fig:exp4}
				\end{figure}
				
				Figure~\ref{fig:exp4} investigates the filter choice \eqref{eq:voc-v-update-full=trap}, \eqref{eq:voc-vhalf-update-vdk}, which is not covered by the second-order result in Theorem~\ref{thm:second-order} because $\Psi_1 \neq 0$ in Lemma~\ref{lem:error-all}. Thus, the perpendicular position error drops to first order in general. The right panel confirms that $\Psi_1$ indeed constitutes the leading-order term: at fixed $\tau = 0.001$, the perpendicular position error follows the graph of 
				\begin{equation} \label{eq:Psi_1_Fig_4}
					\abs{\Psi_1(i\theta)} =  \Bigl|\frac{\sinc(\tfrac{\theta}{2})(\cos(\tfrac{\theta}{2})-1)}{\theta}\Bigr|
				\end{equation}
				across the whole range of $\theta$, including the zeros at the resonances $\theta = 2k\pi$ and the decay for large $\theta$. The left panel shows the convergence of all four error components at the fixed non-resonant value $\tau\normempty{\mgField[]} = 3$. The perpendicular position error is of first order throughout, as predicted. In contrast to the velocity error, however, the order reduction propagates through $\errorelField[n]$ in the error recursion \eqref{eq:voc-discrete}, into the remaining components. These therefore exhibit second-order behavior only as long as the propagated first-order term is negligible, and drop to first order once it becomes dominant visible in the left panel as a transition of the parallel error terms from second to first order. The magnitude of the electric field was adjusted to make the effect visible.
				
				Lastly, we investigate the special case in which the derivative of the electric field $\elField[]'$ commutes with the magnetic field matrix $\mgFieldMatrix[]$. We pick a similar electric potential as for \eqref{eq:E-nonlin}, but varying in the direction of $\mgField[]$ and add an additional confinement in direction orthogonal to $\mgField[]$, i.e.
				\begin{align*}
					V_{\mathrm{c}}(\pos) = \frac{1}{\alpha^2} \sin^2\Bigl(\alpha \pos^{\top}\frac{\mgField[]}{\normempty{\mgField[]}}\Bigr) + \frac{\beta}{2} \normempty{\properp \pos}^2, \qquad \alpha = 10, \qquad \beta = 0.03,
				\end{align*} 
				which leads to the electric field
				\begin{equation} \label{eq:E-comm}
					\elField[c](\pos) = -\nabla V_{\mathrm{c}}(\pos)
					= -
					\frac{1}{\alpha}\sin\Bigl(2\alpha \pos^{\top}\frac{\mgField[]}{\normempty{\mgField[]}}\Bigr) \frac{\mgField[]}{\normempty{\mgField[]}} - \beta \properp \pos.
				\end{equation}
				Its Jacobian is a sum of the parallel and perpendicular projection
				\begin{align*}
					\elFieldJac_c(\pos) = -2\cos\Bigl(2\alpha \pos^{\top}\frac{\mgField[]}{\normempty{\mgField[]}}\Bigr) \propar - \beta \properp
				\end{align*} 
				and hence it commutes with $\mgFieldMatrix[]$ for all $\pos$. Nevertheless, the example is nontrivial as $\elFieldJac_c$ is not a scalar multiple of the identity, $\propar$ or $\properp$ and  $\elField_c''\neq 0$. 
				\begin{figure}
					\includegraphics{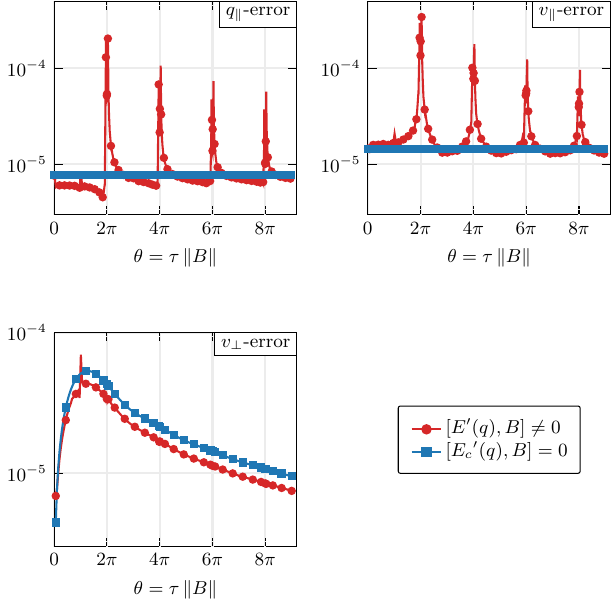}
					\caption{Error components of \eqref{eq:voc-v-update-full=trap}, \eqref{eq:voc-vhalf-update-vdk-neu} versus $\theta=\tau\normempty{\mgField[]}$ for a field pair in which $\elField'$ and $\mgField$ commute (squares) and do not commute (circles). The stepsize $\tau=0.01$ is fixed and $\normempty{\mgField[]}$ is varied.}
					\label{fig:exp5}
				\end{figure}
					
					In Figure~\ref{fig:exp5} we compare the errors for $\elField[c]$ from \eqref{eq:E-comm} to the error for the generic choice \eqref{eq:E-nonlin} and numerically verify Theorem~\ref{thm:commute-no-res}. If $\elField[]'$ commutes with $\mgFieldMatrix[]$ then resonances effects are absent in the parallel errors, which are therefore constant in $\normempty{\mgField[]}$. The perpendicular velocity error shows leading order $\Phi_\perp$ behavior for both choices, but again the resonant spike at $\tau\normempty{\mgField[]} = \pi$ does not exist for the commuting electrical field. Recall that the perpendicular position error does not suffer from resonance effect in the general case, cf.~Figure~\ref{fig:exp2}, hence we do not show it here.
					
					\appendix
					\section{Useful identities}
					
					In this appendix we collect some identities and lemmas used in the proofs above.
					\begin{subequations} \label{eq:id-useful-all}
						\begin{align} 
							\label{eq:id-sinch}
							\sinch (z) &= \tfrac12 \bigl( \varphi_1(z) + \varphi_1(- z)\bigr)= \tfrac12 \varphi_1(\pm z)\bigl(1 + e^{\mp z}\bigr),\\[2mm]
							\label{eq:id-sinch-exp}
							\sinch \bigl(\tfrac{z}{2}\bigr)  &= \varphi_1(\pm z) e^{\mp \tfrac{z}{2}} ,\\[2mm]
							\label{eq:id-phi1-exp}
							\varphi_1(\pm z) e^{\mp z} &= \varphi_1(\mp z)\\
							\label{eq:id-phi1-phi2}
							\varphi_1(z) \varphi_1(-z) &= \varphi_2(z) + \varphi_2(-z)\\[2mm]
							\label{eq:id-phi1phi1half}
							\varphi_1(\pm z) \varphi_1(\mp\tfrac{z}{2}) &= \sinch(\tfrac{z}{2}) \varphi_1(\pm\tfrac{z}{2})\\[2mm]
							\tfrac{z}{2} \sinch(\tfrac{z}{2}) \psi_{1,\cosh}(\tfrac{z}{2}) & = \tfrac12\bigl( \varphi_1(z) \varphi_1(-\tfrac{z}{2}) 
							 -\varphi_1(-z) \varphi_1(\tfrac{z}{2}) \bigr)
							\\[2mm]
							\label{eq:id-psi1sinh}
							\psi_{1,\sinh}(z) &= \frac{\sinh(z)-z}{z^2} 
							= \tfrac1{2} \bigl(\varphi_2(z)- \varphi_2(-z)\bigr)
							\\[2mm]
							\label{eq:id-phi2squaring}
							\varphi_2(z) & = \frac14 \bigl( \varphi_1^2(\tfrac{z}{2}) + 2 \varphi_2(\tfrac{z}{2}) \bigr)
							\\[2mm]
							\label{eq:id-phi3diff}
							\varphi_2(z) \varphi_1(-z) - \varphi_2(-z) & = \varphi_3(z) - \varphi_3(-z)
							\\[2mm]
							\label{eq:id-phi1o2mphi2}
							(\varphi_2(z) - \tfrac12\varphi_1(z))e^{-\tfrac{z}{2}}  & = \frac{\sinch(\tfrac{z}{2}) - \cosh(\tfrac{z}{2})}{z}
						\end{align}
					\end{subequations}

					\begin{lemma} \label{lem:HWL-psipm-phi2-identity}
						For $\psi$ chosen as in \eqref{eq:filter-HL} we have
						\begin{equation*}
							\psi(z) \mp 2 \varphi_1(\pm z) \Upsilon(z)
							= 2\varphi_2(\pm z) , \qquad \Upsilon(z)  = \frac{\sinch^{-1}(z)-1}{z}.
						\end{equation*}
					\end{lemma}
					\begin{proof}
						First, note that by \eqref{eq:tanch-id} and
						\begin{align*}
							\sinch^{-1}(z) & = \frac{2z}{e^z-e^{-z}} = \pm  \frac{2 z e^{\pm z}}{(e^{\pm z}-1)(e^{\pm z}+1)}
						\end{align*}
						yields
						\begin{align*}
							\psi(z) \mp 2 \varphi_1(\pm z) \Upsilon(z) 
							& =   \frac{2}{z^2} \Bigl( \pm z \frac{e^{\pm z}-1}{e^{\pm z}+1} - \bigl(e^{\pm z}-1\bigr) \bigl( \pm \frac{2 z e^{\pm z}}{(e^{\pm z}-1)(e^{\pm z}+1)} -1 \bigr) \Bigr)\\
							& =   \frac{2}{z^2} \Bigl( \pm z \frac{e^{\pm z}-1}{e^{\pm z}+1} \mp \frac{2 z e^{\pm z}}{e^{\pm z}+1} + (e^{\pm z}-1) \Bigr)\\
							& =   \frac{2}{z^2} \bigl( e^{\pm z}-1 \mp z)\\
							& =    2\varphi_2(\pm z).
						\end{align*}
						This proves the identity in the lemma.
					\end{proof}

					\begin{lemma}  \label{lem:HLW-Psi_1}
						For \eqref{eq:HLW_q} we have $\Psi_1= 0$.
					\end{lemma}
					\begin{proof} 
						Using \eqref{eq:id-phi2squaring} and \eqref{eq:id-psi1sinh} we have
						\begin{align*}
							\Psi_1(z) & =
							\tfrac12 \bigl(\varphi_2(\tfrac{z}{2}) - \varphi_2(-\tfrac{z}{2} )\bigr) + \varphi_1(z) \varphi_1(-\tfrac{z}{2} ) 
							- \tfrac12 \varphi_1^2(\tfrac{z}{2}) - \varphi_2(\tfrac{z}{2})
							&& \\
							&=-  \tfrac12 \bigl(\varphi_2(\tfrac{z}{2}) + \varphi_2(-\tfrac{z}{2} )\bigr)
							+ \varphi_1(\tfrac{z}{2}) \bigl( \sinch(\tfrac{z}{2}) - \tfrac12 \varphi_1(\tfrac{z}{2})\bigr)
							&& \text{by \eqref{eq:id-phi1phi1half}}\\
							& = -  \tfrac12 \bigl(\varphi_2(\tfrac{z}{2}) + \varphi_2(-\tfrac{z}{2} )\bigr)
							+ \tfrac12 \varphi_1(\tfrac{z}{2})   \varphi_1(-\tfrac{z}{2})		 
							&&\text{by \eqref{eq:id-sinch}.}
						\end{align*}
						Finally, $\Psi_1=0$  follows from 	\eqref{eq:id-phi1-phi2}.
					\end{proof}

					\section*{Acknowledgments}
					The authors thank Benjamin D\"orich for helpful discussions, Lea Gatzke for her help in coding, and Lea Gatzke and Malik Scheifinger for their careful reading of an earlier version of this manuscript.

\bibliographystyle{amsplain}
\bibliography{refs}

\providecommand{\bysame}{\leavevmode\hbox to3em{\hrulefill}\thinspace}
\providecommand{\MR}{\relax\ifhmode\unskip\space\fi MR }
\providecommand{\MRhref}[2]{%
  \href{http://www.ams.org/mathscinet-getitem?mr=#1}{#2}
}
\providecommand{\href}[2]{#2}
\begin{thebibliography}{10}

\bibitem{Bir18}
C.~K. Birdsall and A.~B. Langdon, \emph{Plasma physics via computer
  simulation}, CRC Press, 2018.

\bibitem{Boris70}
J.~P. Boris, \emph{Relativistic plasma simulation-optimization of a hybrid
  code}, In: Proceeding of Fourth Conference on Numerical Simulations of
  Plasmas (1970), 3--67.

\bibitem{GarSS98}
B.~Garc{\'\i}a-Archilla, J.~M. Sanz-Serna, and R.~D. Skeel,
  \emph{Long-time-step methods for oscillatory differential equations}, SIAM J.
  Sci. Comput. \textbf{20} (1998), no.~3, 930--963.

\bibitem{GriH06}
V.~Grimm and M.~Hochbruck, \emph{Error analysis of exponential integrators for
  oscillatory second-order differential equations}, J. Phys. A: Math. Gen.
  \textbf{39} (2006), no.~19, 5495--5507.

\bibitem{HaiLS22}
E.~Hairer, C.~Lubich, and Y.~Shi, \emph{Large-stepsize integrators for
  charged-particle dynamics over multiple time scales}, Numer. Math.
  \textbf{151} (2022), no.~3, 659--691. \MR{4444838}

\bibitem{HaiLW20}
E.~Hairer, C.~Lubich, and B.~Wang, \emph{A filtered {B}oris algorithm for
  charged-particle dynamics in a strong magnetic field}, Numer. Math.
  \textbf{144} (2020), no.~4, 787--809. \MR{4081132}

\bibitem{HaiL18}
Ernst Hairer and Christian Lubich, \emph{Energy behaviour of the {B}oris method
  for charged-particle dynamics}, BIT \textbf{58} (2018), no.~4, 969--979.
  \MR{3882978}

\bibitem{HocL99}
M.~Hochbruck and C.~Lubich, \emph{A {G}autschi-type method for oscillatory
  second-order differential equations}, Numer. Math. \textbf{83} (1999), no.~3,
  403--426.

\bibitem{HocO10}
M.~Hochbruck and A.~Ostermann, \emph{Exponential integrators}, Acta Numer.
  \textbf{19} (2010), 209--286. \MR{2652783}

\bibitem{LiW22}
T.~Li and B.~Wang, \emph{Geometric continuous-stage exponential
  energy-preserving integrators for charged-particle dynamics in a magnetic
  field from normal to strong regimes}, Applied Numerical Mathematics
  \textbf{181} (2022), 1--22.

\bibitem{NguJT24}
T.~P. Nguyen, I.~Joseph, and M.~Tokman, \emph{Exploring exponential time
  integration for strongly magnetized charged particle motion}, Computer
  Physics Communications \textbf{304} (2024), 109294.

\bibitem{NguJT25}
\bysame, \emph{Nystr{\"o}m type exponential integrators for strongly magnetized
  charged particle dynamics}, Computer Physics Communications \textbf{317}
  (2025), 109848.

\bibitem{QinZXLST13}
H.~Qin, S.~Zhang, J.~Xiao, J.~Liu, Y.~Sun, and W.~M. Tang, \emph{Why is {B}oris
  algorithm so good?}, Physics of Plasmas \textbf{20} (2013), no.~8, 084503.

\bibitem{SchBHL25}
M.~Scheifinger, K.~Busch, M.~Hochbruck, and C.~Lasser, \emph{Time-integration
  of {G}aussian variational approximation for the magnetic {S}chrödinger
  equation}, Journal of Computational Physics \textbf{541} (2025), 114349.

\bibitem{Wan21}
B.~Wang, \emph{Exponential energy-preserving methods for charged-particle
  dynamics in a strong and constant magnetic field}, J. Comput. Appl. Math.
  \textbf{387} (2021), Paper No. 112617, 12. \MR{4199346}

\bibitem{WuW20}
Y.~Wu and B.~Wang, \emph{Explicit symmetric exponential integrators for
  charged-particle dynamics in a strong and constant magnetic field}, Int. J.
  Appl. Comput. Math. \textbf{6} (2020), no.~3, Paper No. 67, 15. \MR{4091998}

\end{thebibliography}
\end{document}